\documentclass[12pt]{amsart}

\usepackage{amsfonts}
\usepackage{amsmath}
\usepackage{amssymb}
\usepackage{amsthm}
\usepackage{enumerate}
\usepackage{geometry}
\usepackage{enumitem}
\usepackage{indentfirst}

\DeclareMathOperator{\lcm}{lcm}
\DeclareMathOperator{\supp}{supp}

\DeclareMathOperator{\ordem}{o}

\newtheorem*{myprop1}{Proposition \ref{prop_subseq}}
\newtheorem*{myprop2}{Proposition \ref{prop_hom}}

\newtheorem{defin}{Definition}[section]
\newtheorem{prop}[defin]{Proposition}
\newtheorem{lem}[defin]{Lemma}
\newtheorem{teo}[defin]{Theorem}
\newtheorem{cor}[defin]{Corollary}

\newtheorem{claim}[defin]{Claim}

\title[Countably compact group topologies with convergent sequences]{Countably compact group topologies on non-torsion Abelian groups of size continuum with non-trivial convergent sequences}

\author[M. K. Bellini, A. C. Boero, I. Castro-Pereira, V. O. Rodrigues, A. H. Tomita]{Matheus Koveroff Bellini, Ana Carolina Boero, Irene Castro Pereira, Vinicius de Oliveira Rodrigues, Artur Hideyuki Tomita}



\thanks{Acknowledgments: The first and fourth authors are doctoral students Proc. FAPESP numbers 2017/15709-6 and  FAPESP 2017/15502-2 who conducted a major revision on the manuscript to its present form. The second author has received financial support from FAPESP (Brazil) --- ``Bolsa de P\'os-Doutorado (processo 2010/19272-2). Projeto: Topologias enumeravelmente compactas em grupos abelianos". The last author received financial support from CNPq (Brazil) --- ``Bolsa de Produtividade em Pesquisa (processo 305612/2010-7). Projeto: Grupos topol\'ogicos, sele\c{c}\~oes e topologias de hiperespa\c{c}o" during the research and has received support from FAPESP Aux\' \i lio regular de pesquisa Proc. Num. 2012/01490-9, CNPq Produtividade em Pesquisa 307130/2013-4 and CNPq Projeto Universal Num. 483734/2013-6 during the research that led to this manuscript. The final revision was made during the support of Aux\' \i lio Regular FAPESP - 2016/26216-8.}

\begin{document}

\maketitle

\begin{abstract}
Under $\mathfrak{p} = \mathfrak{c}$, we  answer Question 24 of \cite{dikranjan&shakhmatov3} for cardinality ${\mathfrak c}$ , by showing that if a non-torsion Abelian group of size continuum admits a countably compact Hausdorff group topology, then it admits a countably compact Hausdorff group topology with non-trivial convergent sequences.
\end{abstract}

\section{Introduction}

\subsection{Some history}

Under Martin's Axiom, Dikranjan and Tkachenko \cite{dikranjan&tkachenko} showed that if $G$ is a non-torsion Abelian group of size continuum, then the following conditions are equivalent:

\begin{itemize}

  \item $G$ admits a countably compact Hausdorff group topology;

  \item $G$ admits a countably compact Hausdorff group topology without non-trivial convergent sequences;

  \item the free rank of $G$ is equal to $\mathfrak{c}$ and, for all $d, n \in \mathbb{N}$ with $d \mid n$, the group $dG[n]$ is either finite or has cardinality $\mathfrak{c}$.

\end{itemize}

Boero and Tomita \cite{boero&tomita3} obtained the same characterization assuming the existence of $\mathfrak{c}$ many incomparable selective ultrafilters (according to the Rudin-Keisler ordering).

In \cite{dikranjan&shakhmatov3}, Dikranjan and Shakhmatov asked (see Question 24) whether an infinite group admitting a countably compact Hausdorff group topology can be endowed with a countably compact Hausdorff group topology that contains a nontrivial convergent sequence. They also asked a similar question for pseudocompact groups and this was solved by
Galindo, Garcia Ferreira and Tomita \cite{galindo&garcia-ferreira&tomita}.
 
 In \cite{galindo&garcia-ferreira&tomita}, it was also noted that if a torsion Abelian group admits a countably compact group topology then it admits a countably compact group topology with a non-trivial convergent sequence.

Boero, Garcia Ferreira and Tomita \cite{boero&garcia-ferreira&tomita} showed that the existence of ${\mathfrak c}$ selective ultrafilters implies that the free Abelian group of cardinality $\mathfrak{c}$ admits a  group topology that makes it countably compact with a non-trivial convergent sequence.

In this article, we give under $\mathfrak{p} = \mathfrak{c}$ a partial answer to Question 24 of \cite{dikranjan&shakhmatov3}, by showing that if a non-torsion Abelian group of size continuum admits a countably compact Hausdorff group topology, then it admits a countably compact Hausdorff group topology with non-trivial convergent sequences.

The construction uses the techniques to construct countably compact group topologies without non-trivial convergent sequences, but to make a sequence convergent is harder and some new ideas are developed to cover some cases that do not appear in free Abelian groups. The sequence used in \cite{boero&garcia-ferreira&tomita} did not work in this construction. The method presented does not guarantee a great number of convergent sequences. One cannot expect to produce convergent subsequences to all sequences, as Tomita \cite{tomita2} showed that there are no sequentially compact group topologies on an infinite free Abelian group.

\subsection{Basic results, notation and terminology}

In what follows, all group topologies are assumed to be Hausdorff. We recall that a topological space $X$ is \emph{countably compact} if every infinite subset of $X$ has an accumulation point.

The following definition was introduced in \cite{bernstein} and is closely related to countable compactness.

\begin{defin}\label{def_p-limit}
Let $p$ be a free ultrafilter on $\omega$ and let $s:\omega\rightarrow X$ be a sequence in a topological space $X$. We say that $x \in X$ is a \emph{$p$-limit point} of $s$ if, for every neighborhood $U$ of $s$, $\{n \in \omega : s(n) \in U\} \in p$. In this case, if $X$ is a Hausdorff space, we write $x = p-\lim s$.
\end{defin}

The set of all free ultrafilters on $\omega$ will be denoted by $\omega^{*}$. It is not difficult to show that a $T_1$ topological space $X$ is countably compact if, and only if, each sequence in $X$ has a $p$-limit point, for some $p \in \omega^{*}$.

\begin{prop}\label{prop_p-limit_product}
If $p \in \omega^{*}$ and $(X_i : i \in I)$ is a family of topological spaces, then $(y_{i})_{i \in I} \in \prod_{i \in I} X_i$ is a $p$-limit point of a sequence $((x_{i}^{n})_{i \in I} : n \in \omega)$ in $\prod_{i \in I} X_i$ if, and only if, $y_i = p-\lim (x_{i}^{n} : n \in \omega)$ for every $i \in I$.
\end{prop}

\begin{prop}\label{prop_p-limit}
Let $G$ be a topological group and $p \in \omega^{*}$.

\begin{enumerate}

  \item If $(x_n : n \in \omega)$ and $(y_n : n \in \omega)$ are sequences in $G$ and $x, y \in G$ are such that $x = p-\lim(x_n : n \in \omega)$ and $y = p-\lim(y_n : n \in \omega)$, then $x + y = p-\lim (x_n + y_n : n \in \omega)$;

  \item If $(x_n : n \in \omega)$ is a sequence in $G$ and $x \in G$ is such that $x = p-\lim(x_n : n \in \omega)$, then $- x = p-\lim(- x_n : n \in \omega)$.

\end{enumerate}
\end{prop}

A \emph{pseudointersection} of a family $\mathcal{G}$ of sets is an infinite set that is almost contained in every member of $\mathcal{G}$. We say that a family $\mathcal{G}$ of infinite sets has the \emph{strong finite intersection property} (SFIP, for short) if every finite subfamily of $\mathcal{G}$ has infinite intersection. The \emph{pseudointersection number} $\mathfrak{p}$ is the smallest cardinality of any $\mathcal{G} \in [\omega]^{\omega}$ with SFIP but with no pseudointersection.

We denote the set of positive natural numbers by $\mathbb{N}$, the integers by $\mathbb{Z}$, the rationals by $\mathbb{Q}$ and the reals by $\mathbb{R}$. The unit circle group $\mathbb{T}$ will be identified with the metric group $(\mathbb{R} / \mathbb{Z}, \delta)$ where $\delta$ is given by $\delta(x + \mathbb{Z}, y + \mathbb{Z}) = \min\{|x - y + a| : a \in \mathbb{Z}\}$ for every $x, y \in \mathbb{R}$. Given a subset $A$ of $\mathbb{T}$, we will denote by $\delta(A)$ the diameter of $A$ with respect to the metric $\delta$. The set of all non-empty open arcs of $\mathbb{T}$ will be denoted by $\mathcal{B}$.

Let $X$ be a set and $G$ be a group. We denote by $G^{X}$ the product $\prod_{x \in X} G_x$ where $G_x = G$ for every $x \in X$. The \emph{support} of $f \in G^{X}$ is the set $\{x \in X : f(x) \neq 0\}$, which will be designated as $\supp f$. The set $\{f \in G^{X} : |\supp f| < \omega\}$ will be denoted by $G^{(X)}$.

The torsion part $T(G)$ of an Abelian group $G$ is the set $\{x \in G : nx = 0 \ \hbox{for some} \ n \in \mathbb{N}\}$. Clearly, $T(G)$ is a subgroup of $G$. For every $n \in \mathbb{N}$, we put $G[n] = \{x \in G : nx = 0\}$. In the case $G = G[n]$, we say that $G$ is \emph{of exponent $n$} provided that $n$ is the minimal positive integer with this property. The order of an element $x \in G$ will be denoted by $\ordem(x)$.

A non-empty subset $S$ of an Abelian group $G$ is said to be \emph{independent} if $0 \not \in S$ and, given distinct elements $s_1, \ldots, s_n$ of $S$ and integers $m_1, \ldots, m_n$, the relation $m_1 s_1 + \ldots + m_n s_n = 0$ implies that $m_i s_i = 0$ for all $i$. This definition implies at once that the group $G$ is a direct sum of cyclic groups if and only if it is generated by an independent subset. The free rank $r(G)$ of $G$ is the cardinality of a maximal independent subset of $G$ such that all of its elements have infinite order. It is easy to verify that $r(G) = |G / T(G)|$ if $r(G)$ is infinite.

The following definition was introduced in \cite{dikranjan&tkachenko}.

\begin{defin}
Let $G$ be an Abelian group and $n \in \mathbb{N} \setminus \{1\}$. A countably infinite subset $S$ of $G$ is said to be \emph{$n$-round} if $nS=\{0\}$ and the restriction of the group homomorphism \[\begin{array}{llll}
                                                  \varphi_{d}: & G & \to & G \\
                                                   & x & \mapsto & d x
                                                \end{array}
  \] to $S$ is finite-to-one for every proper divisor $d$ of $n$.
\end{defin}

Proofs of the following proposition can also be found in \cite{dikranjan&tkachenko}.

\begin{prop}\label{prop_n-round1}
Every infinite set in an Abelian group $G$ of exponent $n$ contains a subset of the form $T + z$, where $z \in G$ and $T$ is a $d$-round subset of $G$ for some divisor $d$ of $n$.
\end{prop}

\begin{prop}\label{prop_n-round2}
Let $G$ be a non-torsion Abelian group that admits a countably compact group topology.
If a subset $S$ of an Abelian group $G$ is $n$-round for some $n$, then $G$ contains an isomorphic copy of the group $\mathbb{Z}_{n}^{(\mathfrak{c})}$.
\end{prop}

We end this section by presenting some notations that will be used throughout this article.

Let $H \in (\mathbb{Q} / \mathbb{Z})^{(P_0 \times \omega)} \oplus \mathbb{Q}^{(P_1)}$. If $x \in P_0 \times \omega$ and $y \in P_1$, then \[H(x) = \frac{p(H, x)}{q(H, x)} + \mathbb{Z}\] where $p(H, x), q(H, x) \in \mathbb{Z}$, $q(H, x) > 0$, $\gcd\{p(H, x), q(H, x)\} = 1$ and $0 \leq p(H, x) < q(H, x)$ and \[H(y) = \frac{p(H, y)}{q(H, y)}\] where $p(H, y), q(H, y) \in \mathbb{Z}$, $q(H, y) > 0$ and $\gcd\{p(H, y), q(H, y)\} = 1$.

 Also, define \[d(H) = \lcm\{q(H, x) : x \in \supp H\}\}.\] If $x \in P_0 \times \omega$ and $y \in P_1$, put \[a(H, x) = p(H, x)   \frac{d(H)}{q(H, x)}\] and \[a(H, y) = p(H, y)   \frac{d(H)}{q(H, y)}.\] Finally, set \[p(H) = \max\{|p(H, y)| : y \in \supp H\}\]  and \[q(H ) = \max\{q(H, y) : y \in \supp H \}.\]

\section{Countably compact group topologies with convergent sequences}
The following theorem is the Proposition 2.4 of \cite{boero&tomita3}.
\begin{prop} \label{megazord}

Let $G$ be an Abelian group such that

\begin{itemize}

  \item $|G| = |G / T(G)| = \mathfrak{c}$, and

  \item $\forall n, d \in \mathbb{N} \ (d \mid n \rightarrow |d G[n]| < \omega \ \hbox{or} \ |d G[n]| = \mathfrak{c})$.

\end{itemize}

Let $\{P_0, P_1\}$ be a partition of $\mathfrak{c}$ such that $|P_0| = |P_1| = \mathfrak{c}$ and $\omega \subset P_1$. Then there exists a group monomorphism $\varphi: G \to (\mathbb{Q} / \mathbb{Z})^{(P_0 \times \omega)} \oplus \mathbb{Q}^{(P_1)}$ such that \[\{(0, \chi_{\xi}) \in (\mathbb{Q} / \mathbb{Z})^{(P_0 \times \omega)} \oplus \mathbb{Q}^{(P_1)} : \xi \in L_1\} \subset \varphi[G]\] and \[\{(y_{\xi}, 0) \in (\mathbb{Q} / \mathbb{Z})^{(P_0 \times \omega)} \oplus \mathbb{Q}^{(P_1)} : \xi \in \cup_{n \in D} L_n\} \subset \varphi[G]\] where $\chi_{\xi} : P_1 \to \mathbb{Q}$ is given by \[\chi_{\xi}(\mu) = \left\{ \begin{array}{lll}
                                                      1 & \hbox{if} & \mu = \xi \\
                                                      0 & \hbox{if} & \mu \neq \xi
                                                    \end{array}
 \right.\] for every $\xi \in P_1$, $L_1 \in [P_1]^{\mathfrak{c}}$ is such that $\omega \subset L_1$, $D$ is the set of all integers $n > 1$ such that $G$ contains an isomorphic copy of the group $\mathbb{Z}_{n}^{(\mathfrak{c})}$ and $\{L_n : n \in D\}$ is a family of pairwise disjoint elements of $[P_0]^{\mathfrak{c}}$ satisfying

\begin{itemize}

  \item $\ordem(y_{\xi}) = n \ \forall \xi \in L_n, n \in D$ and

  \item $\supp y_{\xi} \subset \{\xi\} \times \omega \ \forall \xi \in \cup_{n \in D} L_n$.

\end{itemize}
\end{prop}

We will show that $G$ admits a countably compact Hausdorff group topology with non-trivial convergent sequences. Since $G$ is isomorphic to $\varphi[G]$, it suffices to endow any isomorphic copy of $\varphi[G]$ with such a topology. Therefore, from now on, we will identify $G$ with $\varphi[G]$ and consider $G$ to be a subgroup of $(\mathbb{Q} / \mathbb{Z})^{(P_0 \times \omega)} \oplus \mathbb{Q}^{(P_1)}$ as in Proposition \ref{megazord}.

Our strategy is to construct a group monomorphism $\Phi : (\mathbb{Q} / \mathbb{Z})^{(P_0 \times \omega)} \oplus \mathbb{Q}^{(P_1)} \to \mathbb{T}^{\mathfrak{c}}$ so that $\Phi[G]$ is countably compact and has non-trivial convergent sequences when endowed with the subspace topology induced by $\mathbb{T}^{\mathfrak{c}}$. Such an embedding will be obtained ``coordinate by coordinate" --- that is, for each $\alpha < \mathfrak{c}$ we will construct a group homomorphism $\phi_{\alpha}: (\mathbb{Q} / \mathbb{Z})^{(P_0 \times \omega)} \oplus \mathbb{Q}^{(P_1)} \to \mathbb{T}$ satisfying four significant conditions and $\Phi$ will be the diagonal product of the family $\{\phi_{\alpha} : \alpha < \mathfrak{c}\}$. One of these conditions will guarantee that $\Phi$ is injective, two of them will ensure that $\Phi[G]$ is countably compact and the other one will witness that $\Phi[G]$ has non-trivial convergent sequences.

Each mapping $\phi_{\alpha}$ will be defined in two stages: we will first construct a group homomorphism from a countable subgroup of $(\mathbb{Q} / \mathbb{Z})^{(P_0 \times \omega)} \oplus \mathbb{Q}^{(P_1)}$ into $\mathbb{T}$ by induction and then we will extend it to the whole group $(\mathbb{Q} / \mathbb{Z})^{(P_0 \times \omega)} \oplus \mathbb{Q}^{(P_1)}$. In every inductive step, we will approximate the values of the group homomorphism by non-empty open arcs of $\mathbb{T}$ with suitable properties. To make this possible, we must deal with appropriate families of sequences in $G$ which we will start to sort now.

\begin{defin}\label{def_tipos}
Let $g : \omega \to G \subset (\mathbb{Q} / \mathbb{Z})^{(P_0 \times \omega)} \oplus \mathbb{Q}^{(P_1)}$. Write $g=g_0+g_1$ with $\supp g_0(n)\subset P_0\times\omega$ and $\supp g_1(n)\subset P_1$ for all $n\in\omega$. Also write $g_1=g_{1,0}+g_{1,1}$ in the unique way such that for all $n\in\omega$, $\supp g_{1,0}(n)\subset\omega$ and $\supp g_{1,1}(n)\subset P_1\backslash\omega$.  \\

We say that $g$ is \emph{of type 1} if $\supp g_{1,1}(n) \setminus \cup_{m < n} \supp g_{1,1}(m) \neq \emptyset$ for every $n \in \omega$.

We say that $g$ is \emph{of type 2} if $q(g_{1,1}(n)) > n$ for every $n \in \omega$.

We say that $g$ is \emph{of type 3} if $\{q(g_{1,1}(n)) : n \in \omega\}$ is bounded and $p(g_{1,1}(n)) > n$ for every $n \in \omega$.

We say that $g$ is \emph{of type 4} if $q(g_{1,0}(n)) > n$ for every $n \in \omega$.

We say that $g$ is \emph{of type 5} if there exists $M \in \bigcap_{n \in \omega} \supp g_{1,0}(n)$ such that $\{q(g(n)(M)) : n \in \omega\}$ is bounded and $|p(g(n)(M))| > n$ for every $n \in \omega$.

We say that $g$ is \emph{of type 6} if for each $n \in \omega$, there exists $M_n \in \supp g_{1,0}(n) \setminus \cup_{m < n} \supp g_{1,0}(m)$ such that \[\left\{ \frac{g(n)(M_n)}{M_n !} : n \in \omega \right\} \to 0\] strictly monotonically.

We say that $g$ is \emph{of type 7} if for each $n \in \omega$, there exists $M_n \in \supp g_{1,0}(n) \setminus \cup_{m < n} \supp g_{1,0}(m)$ such that \[\left\{ \frac{g(n)(M_n)}{M_n !} : n \in \omega \right\} \to \xi \in \mathbb{R} \setminus \mathbb{Q}\] strictly monotonically.

We say that $g$ is \emph{of type 8} if for each $n \in \omega$, there exists $M_n \in \supp g_{1,0}(n) \setminus \cup_{m < n} \supp g_{1,0}(m)$ such that \[\left\{ \frac{g(n)(M_n)}{M_n !} : n \in \omega \right\} \to r \in \{- \infty, + \infty\}\] strictly monotonically.

We say that $g$ is \emph{of type 9} if $\left\{ \dfrac{g(n)(M)}{M!} : M \in \supp g_{1,0}(n), n \in \omega \right\}$ is finite and $|\supp g_{1,0}(n)| > n$ for every $n \in \omega$.

We say that $g$ is \emph{of type 10} if $\{q(g_{1,0}(n)) : n \in \omega\}$ and $\{p(g_{1,0}(n)) : n \in \omega\}$ are bounded and $\supp g_{1,0}(n) \setminus \cup_{m < n} \supp g_{1,0}(m) \neq \emptyset$ for every $n \in \omega$.

We say that $g$ is \emph{of type 11} if $q(g_0(n)) > n$ for every $n \in \omega$.

We say that $g$ is \emph{of type 12} if \begin{itemize}

          \item $\supp g_0(n) \setminus \cup_{m < n} \supp g_0(n) \neq \emptyset$ for every $n \in \omega$,

          \item $g_1(n) = 0$ for every $n \in \omega$,

          \item there exists $k \in D$ such that $\ordem(g(n)) = k$ for every $n \in \omega$ and

          \item $\{g(n) : n \in \omega\}$ is a $k$-round subset of $G$.

        \end{itemize}
    
    We let $\mathcal H$ be the set of the constant sequence equal to $0$ and of all sequences of type 1, 2, 3, 4, 5, 6, 7, 8, 9, 10, 11 or 12.
\end{defin}

The set $\mathcal{H}$ enables us not only to ``recover" a subsequence of any sequence in $G$ but also to construct the coordinates $\phi_{\alpha}$ of the embedding $\Phi$. The following two propositions support these statements. Their proofs will be presented in Sections 3 and 4.

\begin{prop}\label{prop_subseq}
If $g: \omega \to G \subset (\mathbb{Q} / \mathbb{Z})^{(P_0 \times \omega)} \oplus \mathbb{Q}^{(P_1)}$, then there exist $l < 2$, $h \in \mathcal{H}$, $c \in G$, $F \in [\omega]^{< \omega}$, $p_i, q_i \in \mathbb{Z}$ with $q_i \neq 0$ for every $i \in F$, $\{M_{n}^{i} : n \in \omega\}\subset \omega$ such that $q_i \mid M_{n}^{i}$ for every $i \in F$ and $j: \omega \to \omega$ strictly increasing such that \[g \circ j(n) = l   h(n) + c - \sum_{i \in F} \frac{p_i}{q_i}   f(M_{n}^{i})\] for every $n \in \omega$, where $f : \omega \to G$ is given by $f(n) = (0, n!   \chi_n)$ for every $n \in \omega$.
\end{prop}

Before stating the next proposition, we fix an enumeration $\{H_{\alpha} : \alpha < \mathfrak{c}\}$ of $(\mathbb{Q} / \mathbb{Z})^{P_0 \times \omega} \oplus \mathbb{Q}^{(P_1)} \setminus \{(0, 0)\}$. We also fix an enumeration $\{h_{\xi} : \xi \in J_1 \cup \bigcup_{n \in D} J_n\}$ of $\mathcal{H}$, where $J_1 \in [L_1]^{\mathfrak{c}}$ and $J_n \in [L_n]^{\mathfrak{c}}$ for every $n \in D$, satisfying the following conditions:

\begin{itemize}

  \item if $\xi \in J_1$, then $h_{\xi}$ is of type $i \in \{1, \ldots, 11\}$ or $h_{\xi}$ is $0$;

  \item if there exists $k \in D$ such that $\xi \in J_k$, then $h_{\xi}$ is of type 12 and $\ordem(h_{\xi}(n)) = k$ for every $n \in \omega$;

  \item $\bigcup_{n \in \omega} \supp h_{\xi}(n) \subset (\xi \times \omega) \cup \xi$ for every $\xi \in J_1 \cup \bigcup_{n \in D} J_n$.

\end{itemize}

We may write $h_\xi=h_\xi^0+h^{1, 0}_\xi+h^{1, 1}_\xi$ in the only way such that $\supp h^0_\xi\subset P_0\times \omega$, $\supp h^{1, 0}_\xi\subset \omega$, $\supp h^{1, 1}_\xi\subset P_1\setminus \omega$, and let $h^1_\xi=h^{1, 0}_\xi+h^{1, 1}_\xi$.

\begin{prop}[$\mathfrak{p} = \mathfrak{c}$]\label{prop_hom}
For each $\alpha < \mathfrak{c}$ and each $\xi \in J_1 \cup \bigcup_{n \in D} J_n$ there exists $S_{\xi, \alpha} \in [\omega]^{\omega}$ such that if $\alpha < \beta < \mathfrak{c}$, then $S_{\xi, \beta} \subset^{*} S_{\xi, \alpha}$. There also exists a group homomorphism $\phi_{\alpha} : (\mathbb{Q} / \mathbb{Z})^{(P_0 \times \omega)} \oplus \mathbb{Q}^{(P_1)} \to \mathbb{T}$ satisfying the following conditions:

\begin{enumerate}[label=(\roman*)]

  \item $\phi_{\alpha}(H_{\alpha}) \neq 0 + \mathbb{Z}$;

  \item if $\xi \in J_1$, then the sequence $(\phi_{\alpha}(h_{\xi}(n)) : n \in S_{\xi, \alpha})$ converges to $\phi_{\alpha}(0, \chi_{\xi})$;

  \item if $\xi \in \bigcup_{n \in D} J_n$, then the sequence $(\phi_{\alpha}(h_{\xi}(n)) : n \in S_{\xi, \alpha})$ converges to $\phi_{\alpha}(y_{\xi}, 0)$;

  \item for each $p \in \mathbb{Z} \setminus \{0\}$, the sequence $\left( \phi_{\alpha} \left( 0, \frac{1}{p}   n!   \chi_{n} \right) : n \in \omega \right)$ converges to $0 + \mathbb{Z}$.

\end{enumerate}
\end{prop}

We end this section by showing how Propositions \ref{prop_subseq} and \ref{prop_hom} can be used to endow a non-torsion Abelian group of size continuum that admits a countably compact Hausdorff group topology with a countably compact Hausdorff group topology containing non-trivial convergent sequences.

\begin{teo}[$\mathfrak{p} = \mathfrak{c}$]
If $G$ is an Abelian group such that $|G| = |G / T(G)| = \mathfrak{c}$ and $\forall n, d \in \mathbb{N} \ (d \mid n \rightarrow |d G[n]| < \omega \ \hbox{or} \ |d G[n]| = \mathfrak{c})$, then $G$ admits a countably compact Hausdorff group topology with non-trivial convergent sequences.
\end{teo}

\begin{proof}
It follows from Proposition \ref{prop_hom} (i) that \[\begin{array}{cccc}
                                                        \Phi : & (\mathbb{Q} / \mathbb{Z})^{(P_0 \times \omega)} \oplus \mathbb{Q}^{(P_1)} & \to & \mathbb{T}^{\mathfrak{c}} \\
                                                               & H & \mapsto & \Phi(H)
                                                      \end{array}
\] given by $\Phi(H)(\alpha) = \phi_{\alpha}(H)$ for every $\alpha < \mathfrak{c}$ is a group monomorphism. Thus, $\Phi[G]$ is isomorphic to $G$ and since $\mathbb{T}^{\mathfrak{c}}$ is a Hausdorff topological group, the subspace topology induced by $\mathbb{T}^{\mathfrak{c}}$ turns $\Phi[G]$ into a Hausdorff topological group.

It follows from Proposition \ref{prop_hom} (iv) that if $p \in \mathbb{Z} \setminus \{0\}$, then the sequence $\left( \Phi \left( 0, \frac{1}{p}   n!   \chi_{n} \right) : n \in \omega \right)$ converges to $0$.

Let $g: \omega \to \Phi[G]$. It follows from Proposition \ref{prop_subseq} that there exist $l < 2$, $\xi \in J_1 \cup \bigcup_{n \in D} J_n$, $c \in G$, $F \in [\omega]^{< \omega}$, $p_i, q_i \in \mathbb{Z}$ with $q_i \neq 0$ for every $i \in F$, $\{M_{n}^{i} : n \in \omega\}$ such that $q_i \mid M_{n}^{i}$ for every $i \in F$ and a strictly increasing function $j: \omega \to \omega$ such that \[\Phi^{-1} \circ g \circ j(n) = l   h_{\xi}(n) + c - \sum_{i \in F} \frac{p_i}{q_i}   f(M_{n}^{i})\] for every $n \in \omega$, where $f : \omega \to G$ is given by $f(n) = (0, n!   \chi_n)$ for every $n \in \omega$.

Fix $p_{\xi} \in \omega^{*}$ containing $\{S_{\xi, \alpha} : \alpha < \mathfrak{c}\}$. According to Proposition \ref{prop_hom} (ii) and (iii), the sequence $(\phi_{\alpha}(h_{\xi}(n)) : n \in S_{\xi, \alpha})$ converges to $\phi_{\alpha}(0, \chi_{\xi})$ (respectively, $\phi_{\alpha}(y_{\xi}, 0)$) if $\xi \in L_1$ (respectively, $\xi \in \bigcup_{n \in D} J_n$) for every $\alpha < \mathfrak{c}$. Thus, if $\xi \in J_1$, then \[\phi_{\alpha}(0, \chi_{\xi}) = p_{\xi}-\lim (\phi_{\alpha}(h_{\xi}(n)) : n \in \omega)\] for every $\alpha < \mathfrak{c}$ and if $\xi \in \bigcup_{n \in D} J_n$, then \[\phi_{\alpha}(y_{\xi}, 0) = p_{\xi}-\lim (\phi_{\alpha}(h_{\xi}(n)) : n \in \omega)\] for every $\alpha < \mathfrak{c}$. It follows from Proposition \ref{prop_p-limit_product} that $\xi \in J_1$, then \[\Phi(0, \chi_{\xi}) = p_{\xi}-\lim (\Phi(h_{\xi}(n)) : n \in \omega)\] and if $\xi \in \bigcup_{n \in D} J_n$, then \[\Phi(y_{\xi}, 0) = p_{\xi}-\lim (\Phi(h_{\xi}(n)) : n \in \omega).\]

Since the sequence $\left( \Phi \left( 0, \frac{1}{q_i}   n!   \chi_{n} \right) : n \in \omega \right)$ converges to $0$ for each $i \in F$, it follows from Proposition \ref{prop_p-limit} that if $\xi \in J_1$, then \[\Phi(l   (0, \chi_{\xi}) + c) = p_{\xi}-\lim \left( \Phi \left( l   h_{\xi}(n) + c - \sum_{i \in F} \frac{p_i}{q_i}   f(M_{n}^{i}) \right) : n \in \omega \right)\] and if $\xi \in \bigcup_{n \in D} J_n$, then \[\Phi(l   (y_{\xi}, 0) + c) = p_{\xi}-\lim \left( \Phi \left( l   h_{\xi}(n) + c - \sum_{i \in F} \frac{p_i}{q_i}   f(M_{n}^{i}) \right) : n \in \omega \right).\] Consequently, $\Phi(l   (0, \chi_{\xi}) + c)$ (respectively, $\Phi(l   (y_{\xi}, 0) + c)$) is an accumulation point of $(g(n) : n \in \omega)$ if $\xi \in J_1$ (respectively, if $\xi \in \bigcup_{n \in D} J_n$).
\end{proof}

Since $\mathfrak{p} = \mathfrak{c}$ implies that an Abelian group $G$ satisfying $|G| = |G / T(G)| = \mathfrak{c}$ and $\forall n, d \in \mathbb{N} \ (d \mid n \rightarrow |d G[n]| < \omega \ \hbox{or} \ |d G[n]| = \mathfrak{c})$ admits a countably compact Hausdorff group topology (see, for instance, \cite{boero&tomita3}), the following corollary holds.

\begin{cor}[$\mathfrak{p} = \mathfrak{c}$]
If $G$ is a non-torsion Abelian group of size continuum that admits a countably compact Hausdorff group topology, then $G$ admits a countably compact Hausdorff group topology with non-trivial convergent sequences.
\end{cor}

\section{Proof of Proposition \ref{prop_subseq}}

\begin{myprop1}
If $g: \omega \to G \subset (\mathbb{Q} / \mathbb{Z})^{(P_0 \times \omega)} \oplus \mathbb{Q}^{(P_1)}$, then there exist $l < 2$, $h \in \mathcal{H}$, $c \in G$, $F \in [\omega]^{< \omega}$, $p_i, q_i \in \mathbb{Z}$ with $q_i \neq 0$ for every $i \in F$, $\{M_{n}^{i} : n \in \omega\}$ such that $q_i \mid M_{n}^{i}$ for every $i \in F$ and $j:\omega\rightarrow \omega$ strictly increasing such that \[g\circ j(n)= l   h(n) + c - \sum_{i \in F} \frac{p_i}{q_i}   f(M_{n}^{i})\] for every $n \in \omega$, where $f : \omega \to G$ is given by $f(n) = (0, n!   \chi_n)$ for every $n \in \omega$.
\end{myprop1}

\begin{proof}
Write $g=g_0+g_{1, 0}+g_{1, 1}$ in the only way such that $\supp g_0\subset P_0\times \omega$, $\supp g_{1, 0}\subset \omega$, $\supp g_{1, 1}\subset P_1\setminus \omega$, and let $g_1=g_{1, 0}+g_{1, 1}$. There exist nine cases to consider.

In the first three cases, we assume $(\ddag)$: there exists $J_0 \in [\omega]^\omega$ such that $g_1|J_0$ is constant.

\begin{description}[style=unboxed,leftmargin=0cm]

  \item[Case 1] $(\ddag)$ and $\{q(g_0(n)): n \in J_0\}$ is unbounded. \\ By induction, define a strictly increasing sequence $(n_k : k \in \omega)$ of elements of $J_0$ such that $q(g_0(n_k)) > k$ for each $k \in \omega$. The function $j: \omega \to \omega$ defined by $j(k) = n_k$ for each $k \in \omega$ is strictly increasing and $g \circ j$ is of type 11.

  \item[Case 2] $(\ddag)$ and $\{q(g_0(n)): n \in J_0\}$ is bounded and $\bigcup_{n \in J_0} \supp g_0(n) $ is infinite. \\ Since $\bigcup_{n \in J_0} \supp g_0(n) $ is infinite, there exists $J_1 \in [J_0]^{\omega}$ such that $\supp g_0(n) \setminus \cup_{m < n, m \in J_1} \supp g_0(m) \neq \emptyset$ for every $n \in J_1$. We have that $\{g(n) : n \in J_1\}$ is an infinite subset of $G$ and so is $S = \{g(n) - g(n_0) : n \in J_1\}$, where $n_0 \in J_1$ is a fixed element. Since $\{q(g_0(n)) : n \in J_1\}$ is bounded, there exists a natural number $k > 1$ such that $kS=\{0\}$. Thus, there exist $T \subset G$ infinite and $z \in G$ such that $T + z \subset S$ and $T$ is $d$-round, for some divisor $d$ of $k$. Let $J_2 \in [J_1]^{\omega}$ such that $T + z = \{g(n) - g(n_0) : n \in J_2\}$. In other words, $T = \{g(n) - g (n_0) - z : n \in J_2\}$. $T$ is $d$-round, so $dT=\{0\}$. Therefore, there exist $J_3 \in [J_2]^{\omega}$ and $r$ a divisor of $d$ such that $\ordem(g(n) - g(n_0) - z) = r$ for every $n \in J_3$. Observe that $r \in D$, since $\{g (n) - g (n_0) - z : n \in J_3\}$ is $r$-round. Put $c = - g (n_0) - z$ and recursively define $j:\omega\rightarrow J_3$ strictly increasing such that, $\supp (g \circ j(n) - c) \setminus \cup_{m < n} \supp (g \circ j(m) - c) \neq \emptyset$ for every $n \in \omega$. Note that the sequence $(g \circ j(n) + c : n \in \omega)$ is of type 12.

  \item[Case 3]$(\ddag)$ and $\{q(g_0(n) : n \in J_0)\}$ is bounded and $\bigcup_{n \in J_0} \supp g_0(n)$ is finite. \\ There exists a strictly increasing function $j : \omega \to J_0$ such that $g \circ j(n)$ is constant.

\end{description}

 For the next three cases  there will be no extra assumptions.
\begin{description}[style=unboxed,leftmargin=0cm]

  \item[Case 4] $\bigcup_{n \in \omega} \supp g_{1,1}(n)$ is infinite. \\ By induction, define a strictly increasing sequence $(n_k : k \in \omega)$ of natural numbers such that $\supp g_{1,1}(n_{k + 1}) \setminus \cup_{l < k + 1} \supp g_{1,1}(n_l)  \neq \emptyset$. The function $j : \omega \to \omega$ defined by $j(k) = n_k$ for each $k \in \omega$ is strictly increasing and $g \circ j$ is of type 1.

  \item[Case 5] $\bigcup_{n \in \omega} \supp g_{1,1}(n) $ is finite and $\{q(g_{1,1}(n) ) : n \in \omega\}$ is unbounded. \\ By induction, define a strictly increasing sequence $\{n_k : k \in \omega\}$ of natural numbers such that $q(g_{1,1}(n_k) ) > k$ for each $k \in \omega$. The function $j : \omega \to \omega$ defined by $j(k) = n_k$ for each $k \in \omega$ is strictly increasing and $g \circ j$ is of type 2.

  \item[Case 6] $\bigcup_{n \in \omega} \supp g_{1,1}(n) $ is finite, $\{q(g_{1,1}(n) ) : n \in \omega\}$ is bounded and $\{p(g_{1, 1}(n)) : n \in \omega\}$ is unbounded. \\ By induction, define a strictly increasing sequence $(n_k : k \in \omega)$ of natural numbers such that $p(g_{1,1}(n_k)) > k$ for each $k \in \omega$. The function $j : \omega \to \omega$ defined by $j(k) = n_k$ for each $k \in \omega$ is strictly increasing and $g \circ j$ is of type 3.

\end{description}

For the last three cases, we shall consider that $\bigcup_{n \in \omega} \supp g_{1,1}(n)$ is finite and that $\{q(g_{1,1}(n)) : n \in \omega\}$ and $\{p(g_{1,1}(n)) : n \in \omega\}$ are bounded. So, there exists $J_0\in [\omega]^{\omega}$ such that $g_{1,1}|_{J_0}$ is constant.

\begin{description}[style=unboxed,leftmargin=0cm]

  \item[Case 7] $\{q(g_{1,0}(n)) : n \in J_0\}$ is unbounded. \\ By induction, define a strictly increasing sequence $(n_k : k \in \omega)$ of elements of $J_0$ such that $q(g_{1,0}(n_k)) > k$ for each $k \in \omega$. The function $j: \omega \to \omega$ defined by $j(k) = n_k$ for each $k \in \omega$ is strictly increasing and $g \circ j$ is of type 4.

  \item[Case 8] $\{q(g_{1, 0}(n)) : n \in J_0\}$ is bounded and $\{p(g_{1, 0}(n)) : n \in J_0\}$ is unbounded. \\ Let \[Q = \left\{ \frac{g(n)(M)}{M!} : M \in \supp g_{1,0}(n), n \in J_0 \right\}.\]

      \begin{description}[style=unboxed,leftmargin=0cm]

        \item[Q is infinite] There exist $J_1 \in [J_0]^{\omega}$ and $(M_n : n \in J_1)$ such that $M_n \in \supp g_{1, 0}(n)$ and \[\left( \frac{g(n)(M_n)}{M_n!} : n \in J_1 \right) \to r\] strictly monotonically for some $r \in [- \infty, + \infty]$. We can suppose that $M_n \not \in \cup_{m < n, m \in J_1} \supp g_{1, 0}(m)$ for every $n \in J_1$ or that $\{M_n : n \in J_1\}$ is a singleton. Suppose that the first case occurs. If $r = 0$, then $g \circ j$ is of type 6, where $j: \omega \to J_1$ is the order isomorphism. Analogously, if $r \in \mathbb{R} \setminus \mathbb{Q}$, then $g \circ j$ is of type 7 and if $r \in \{- \infty, + \infty\}$, then $g \circ j$ is of type 8. If $r \in \mathbb{Q} \setminus \{0\}$, let $\tilde{g} : \omega \to (\mathbb{Q} / \mathbb{Z})^{(P_0 \times \omega)} \oplus \mathbb{Q}^{(P_1)}$ be such that $\tilde{g}(n) = g (n) - r   f(M_n)$ for every $n \in J_1$. Since $(M_n : n \in J_1)$ is injective, there exists a cofinite subset $J_2$ of $J_1$ such that $\tilde{g}(n) \in G$ for every $n \in J_2$. Note that there exists $J_3 \in [J_2]^{\omega}$ such that \[\left( \frac{\tilde{g}(n)(M_n)}{M_n!} : n \in J_3 \right) \to 0\]
        strictly monotonically. Let $j : \omega \to J_3$ be an order isomorphism. Then $\tilde{g} \circ j$ is of type 6 and $g \circ j(n) = \tilde{g} \circ j(n) + r   f(M_n)$ for every $n \in \omega$. Now, suppose that there exists $M \in \omega$ such that $M_n = M$ for every $n \in J_1$. Since \[\left( \frac{g(n)(M)}{M!} : n \in J_1 \right)\] is injective and $\{q(g(n), M) : n \in J_1\}$ is bounded, there exists $J_2 \in [J_1]^{\omega}$ such that $|p(g(n), M)| > n$ for every $n \in J_2$. Hence, $g \circ j$ is of type 5, where $j : \omega \to J_2$ is the order isomorphism .

            \item[Q is finite] Either there exists a strictly growing sequence $u:\omega\rightarrow J_0$ such that $|\supp g_{1, 0}(u(n))| > n$ for every $n \in \omega$ or there exists $J_1\in [J_0]^\omega$ such that $\{|\supp g_{1, 0}(n)| : n \in J_1\}$ is a singleton. In the first case, $g \circ u$ is of type 9. Suppose that there exists $k \in \omega$ such that $|\supp g_{1, 0}(n)| = k$ for every $n \in J_1$ and write $\supp g_{1, 0}(n)  = \{M_{0}^{n}, \ldots, M_{k - 1}^{n}\}$, where $M_{i}^{n} \neq M_{i'}^{n}$ if $i \neq i'$. Since $Q$ is finite, there exist $J_2 \in [J_1]^{\omega}$ and $p_0/q_0, \ldots, p_{k - 1}/q_{k - 1} \in \mathbb{Q} \setminus \{0\}$ such that \[\frac{g (n)(M_{i}^{n})}{M_{i}^{n}!} = \frac{p_i}{q_i}\]
            
            for all $n \in J_2$ and $i < k$. By refining $J_2$ (if necessary), we can suppose that for each $i < k$, $(M_{i}^{n} : n \in \omega)$ is either constant or injective with $q_i \leq M_{i}^{n}$ for every $n \in J_2$. Let $K_0 = \{i < k : \{M_{i}^{n} : n \in J_2\} \ \hbox{is a constant sequence}\}$ and $K_1 = k \setminus K_0$. Let $\hat{g} : \omega \to (\mathbb{Q} / \mathbb{Z})^{(P_0 \times \omega)} \oplus \mathbb{Q}^{(P_1)}$ be such that \[\hat{g}(n) = g(n) - \sum_{i \in K_1} \frac{p_i}{q_i} f(M_{i}^{n})\] for every $n \in J_2$. Observe that $\hat{g}(n) \in G$ for every $n \in J_2$ and that $\hat{g}_{1, 0}|_{J_2}$ is a constant sequence. Since $\hat{g}_{1,1}(n) = g_{1, 1}(n)$ for every $n \in J_2$, we conclude that $(\hat{g}_1(n): n \in J_2)$ is constant. If $\{q(\hat{g}_{0}(n)) : n \in J_2\}$ is unbounded, we proceed as in Case 1. If $\{q(\hat{g}_0(n)) : n \in J_2\}$ is bounded and $\bigcup_{n \in J_2} \supp \hat{g}_{0}(n)$ is infinite, we proceed as in Case 2. Finally, if $\{q(\hat{g}_{0}(n) ) : n \in J_2\}$ is bounded and $\bigcup_{n \in J_2} \supp \hat{g}_0(n) $ is finite, we proceed as in Case 3.

  \item[Case 9] $\{q(g_{1,0}(n) : n \in J_0\}$ and $\{p(g_{1,0}(n)) : n \in J_0\}$ are bounded and $\bigcup_{n \in J_0} \supp g_{1,0}(n)$ is infinite. \\ By induction, define a strictly increasing sequence $(n_k : k \in \omega)$ of elements of $J_0$ such that $\supp g_{1,0}(n_{k + 1}) \setminus \cup_{l < k + 1} \supp g_{1,0}(n_l) \neq \emptyset$. The function $j: \omega \to \omega$ defined by $j(k) = n_k$ for each $k \in \omega$ is strictly increasing and $g \circ j$ is of type 10.
      \end{description}
\end{description}

Now, if cases 6 to 9 do not hold then $g_{1,1}|_{J_0}$ is constant, $\bigcup_{n \in J_0} \supp g_{1,0}(n)$ is finite and $\{q(g_{1,0}(n)) : n \in J_0\}$ and $\{p(g_{1,0}(n)) : n \in J_0\}$ are bounded. Therefore, there exists a strictly increasing function $j: \omega \to \omega$ such that $g_{1,0} \circ j$ is constant and $(\ddag)$ is satisfied.
\end{proof}

\section{Proof of Proposition \ref{prop_hom}}

This section is devoted to prove Proposition \ref{prop_hom}. As mentioned, we will concern ourselves primarily with the construction of group homomorphisms from countable subgroups of $(\mathbb{Q} / \mathbb{Z})^{(P_0 \times \omega)} \oplus \mathbb{Q}^{(P_1)}$ into $\mathbb{T}$. These countable subgroups will be of the form $(\mathbb{Q} / \mathbb{Z})^{((P_0 \cap E) \times \omega)} \oplus \mathbb{Q}^{(P_1 \cap E)}$ for some $E \in [\mathfrak{c}]^{\omega}$ and the following proposition ensures the existence of a suitable $E$ and can be proved by an easy closing off argument.

\begin{prop}\label{prop_const_E}
Let $\alpha < \mathfrak{c}$. There exists $E \in [\mathfrak{c}]^{\omega}$ such that

\begin{enumerate}[label=(\roman*)]

  \item $\supp H_{\alpha} \subset (E \times \omega) \cup E$,
  \item $\omega \subset E$,

  \item $|E \cap (J_1 \cup \bigcup_{n \in D} J_n)| = \omega$ and

  \item $\bigcup_{n \in \omega} \supp h_{\xi}(n) \subset (E \times \omega) \cup E$ for every $\xi \in E \cap (J_1 \cup \bigcup_{n \in D} J_n)$.

\end{enumerate}
\end{prop}

The next two lemmas will be necessary in the successor step of the induction in Lemma \ref{lem_const_hom}. A proof of Lemma \ref{lem_kronecker_T2} can be found in \cite{boero&tomita2}.

\begin{lem}\label{lem2}
Let $c, d \in \mathbb{Z} \setminus \{0\}$, $\epsilon > 0$, $a \in \mathbb{T}$ and $B = \mathbb{T}$ or $B \in \mathcal{B}$ be such that $\delta(B) \geq \epsilon$ and $|d|   \epsilon > \gcd(c, d)$. There exists $x \in \mathbb{T}$ such that $d   x = a$ and $c   x \in B$.
\end{lem}

\begin{proof}
	let $e=\gcd(c, d)$ and let $c', d' \in \mathbb Z$ be such that $c=ec'$ and $d=ed'$. Notice that $\epsilon>\frac{1}{|d'|}$. If $B=\mathbb T$ the proof is trivial, so suppose $\delta(B)\geq \epsilon$ and $|d|\epsilon>e$.
	
	Let $a=\tilde a+\mathbb Z$ for some $\tilde a \in \mathbb R$. Since $\epsilon>\frac{1}{|d'|}$, there exists $l \in \mathbb Z$ such that:
	
	$$\frac{c\tilde a}{d}+\frac{l}{d'} +\mathbb Z \in B.$$
	
	Since $\gcd(c', d')=1$, there exists $u, v \in \mathbb Z$ such that $uc'+vd'=l$. Now, since $\frac{c}{d}=\frac{c'}{d'}$, it follows that:
	
		$$\frac{c\tilde a}{d}+\frac{uc}{d} +\mathbb Z \in B.$$
		
	So let $x=\frac {\tilde a+u}{d}+\mathbb Z$ and we are done.
\end{proof}

Kronecker's Lemma says that if $\{1, \xi_0, \ldots, \xi_{k - 1}\}$ is a linearly independent subset of the vector space $\mathbb{R}$ over the
field $\mathbb{Q}$, then $\{(\xi_0 n, \ldots, \xi_{k - 1} n) + \mathbb{Z}^{k} : n \in \mathbb{Z}\}$ is a dense subset of the product of $k$ tori (see, for instance, \cite{brocker&tomdieck}). If $\xi \in \mathbb{R} \setminus \mathbb{Q}$, then $\{(x, \xi x) + \mathbb{Z}^{2} : x \in \mathbb{R}\}$ is a dense subset of the torus $\mathbb{T}^{2} = \mathbb{R}^{2} / \mathbb{Z}^{2}$. Thus, for each $\epsilon > 0$, there exists $l(\epsilon, \xi) > 0$ such that if $I \subset \mathbb{R}$ is an interval of length greater than $l(\epsilon, \xi)$, then $\{(x, \xi x) + \mathbb{Z}^{2} : x \in I\}$ is $\epsilon$-dense in $\mathbb{T}^{2}$.

\begin{lem}\label{lem_kronecker_T2}
Let $c_1, c_2 \in \mathbb{Z} \setminus \{0\}$, $\epsilon > 0$ and $B_1, B_2 \in \mathcal{B}$ be such that $\delta(B_1) \geq \epsilon$ and $\delta(B_2) \geq \epsilon$. Suppose that there exist $\xi \in \mathbb{R} \setminus \mathbb{Q}$ and $a > l(\epsilon / 4, \xi)$ such that $\left| \frac{c_1}{c_2} a - \xi a \right| < \frac{\epsilon}{4}$. There exists $x \in \mathbb{T}$ such that $c_1   x \in B_1$ and $c_2   x \in B_2$.
\end{lem}

\begin{lem}\label{lem_const_hom}
	Let $\alpha < \mathfrak{c}$ and $E \in [\mathfrak{c}]^{\omega}$ be as in Proposition \ref{prop_const_E}. For each $\xi \in E \cap (J_1 \cup \bigcup_{n \in D} J_n)$, let $R_{\xi} \in [\omega]^{\omega}$. There exists a group homomorphism  $\phi_{\alpha, E} : (\mathbb{Q} / \mathbb{Z})^{((P_0 \cap E) \times \omega)} \oplus \mathbb{Q}^{(P_1 \cap E)} \to \mathbb{T}$ satisfying the following conditions:
	
	\begin{enumerate}[label=(\roman*)]
		
		\item $\phi_{\alpha, E}(H_{\alpha}) \neq 0 + \mathbb{Z}$;
		
		\item for each $\xi \in E \cap J_1$, there exists $S_{\xi} \in [R_{\xi}]^{\omega}$ such that the sequence $(\phi_{\alpha, E}(h_{\xi}(n)) : n \in S_{\xi})$ converges to $\phi_{\alpha, E}(0, \chi_{\xi})$;
		
		\item for each $\xi \in E \cap \bigcup_{n \in D} J_n$, there exists $S_{\xi} \in [R_{\xi}]^{\omega}$ such that the sequence $(\phi_{\alpha, E}(h_{\xi}(n)) : n \in S_{\xi})$ converges to $\phi_{\alpha, E}(y_{\xi}, 0)$;
		
		\item for each $p \in \mathbb{Z} \setminus \{0\}$, the sequence $\left( \phi_{\alpha, E} \left( 0, \frac{1}{p}   n!   \chi_{n} \right) : n \in \omega \right)$ converges to $0 + \mathbb{Z}$.
		
	\end{enumerate}
\end{lem}

\begin{proof}
	Denote by $\mathbb{P}$ the set of all prime numbers. Let $\{p_n : n \in \omega\}$ be an enumeration of $\mathbb{P}$ such that $|\{n \in \omega : p = p_n\}| = \omega$ for every $p \in \mathbb{P}$. We may write $H_\alpha=H_0+H_{1, 0}+H_{1, 1}$ in the only way such that $\supp H_0\subset P_0\times \omega$, $\supp H_{1, 0}\subset \omega$, $\supp H_{1, 1}\subset P_1\setminus \omega$, and let $H_1=H_{1, 0}+H_{1, 1}$.
	
	Let $\{\theta_n : n \in \omega\}$ be an enumeration of $E \cap (J_1 \cup \bigcup_{n \in D} J_n)$ such that $|\{n \in \omega : \xi = \theta_n\}| = \omega$ for every $\xi \in E \cap (J_1 \cup \bigcup_{n \in D} J_n)$. Let also $\{e_{n}^{0} : n \in \omega\}$ be an enumeration of $(P_0 \cap E) \times \omega$ and $\{e_{n}^{1} : n \in \omega\}$ be an enumeration of $P_1 \cap E$.
	
	 For each $m\in\omega$, let \[A_m =  \begin{cases}
	\emptyset & \text{if }  \theta_m \in J_1 \\
	\supp y_{\theta_m} & \text{if }  \theta_m \in \bigcup_{n \in D} J_n.
	\end{cases} \]
	
	Let $\Lambda_{(\xi, n), k} : (P_0 \cap E) \times \omega \to \mathbb{T}$ be given by \[\Lambda_{(\xi, n), k}(\mu, l) = \left\{ \begin{array}{lll}
	1 / k + \mathbb{Z} & \hbox{if} & (\mu, l) = (\xi, n) \\
	0 + \mathbb{Z} & \hbox{if} & (\mu, l) \neq (\xi, n)
	\end{array}
	\right.\] for all $(\xi, n) \in (P_0 \cap E) \times \omega$ and $k \in \mathbb{Z} \setminus \{0\}$.
	
	Let $c_m$, for each $m\in\omega$, be given by \[c_{0} = \left\{ \begin{array}{lll}
	d(H_{0}) & \hbox{if} & \theta_{0} \in L_1 \\
	d(H_{0})   d(y_{\theta_{0}}, 0) & \hbox{if} & \theta_{0} \in \bigcup_{n \in D} L_n,
	\end{array}
	\right.\] \[c_{m + 1} = \left\{ \begin{array}{lll}
	c_m & \hbox{if} & \theta_{m + 1} \in L_1 \\
	c_m   d(y_{\theta_{m + 1}}, 0) & \hbox{if} & \theta_{m + 1} \in \bigcup_{n \in D} L_n.
	\end{array}
	\right.\] 
	
	Set $G_{0}^{0} = \supp H_0 \cup \{e_{0}^{0}\} \cup A_0$ and let $\phi_{0} : \langle \{ \Lambda_{(\theta, m), c_0} : (\theta, m) \in G_{0}^{0}\} \rangle \to \mathbb{T}$ be a group homomorphism such that $\phi_{0}(H_{0}) \neq 0 + \mathbb{Z}$, if $H_{0} \neq 0$.

	Set also $G_{0}^{1} = \supp H_1 \cup \{e_{0}^{1}\} \cup B_0$, where \[B_m = \left\{ \begin{array}{lll}
	\{\theta_m\} & \hbox{if} & \theta_m \in J_1 \\
	\emptyset & \hbox{if} & \theta_m \in \bigcup_{n \in D} J_n
	\end{array} \right.
	\] for every $m \in \omega$ and let $\psi_{0} : G_{0}^{1} \to \mathcal{B}$ be such that
	
	\begin{itemize}
		
		\item $\delta(\psi_{0}(\xi)) = r_0 / d(H_{1})$ for every $\xi \in G_{0}^{1}$, where $r_0 = 1 / (4   \sum_{\xi \in P_1} |a(H_{\alpha}, \xi)|)$ and
		
		\item $0 + \mathbb{Z} \not \in \phi_{0}(H_0) + \sum_{\xi \in P_1} a(H_{1}, \xi)   \psi_{0}(\xi)$, if $H_1 \neq 0$.
		
	\end{itemize}
	
	This concludes the first step of the induction. Now, we will start the successor stage. Fix $m \in \omega$ and suppose we have defined $s_0=d(H_1)$, $b_{-1} = 0$, $b_{m - 1} \in R_{\theta_{m - 1}}$ (if $m \geq 1$), $r_m > 0$, $s_m>0$, $G_{m}^{0} \in [P_0 \times \omega]^{< \omega}$, $G_{m}^{1} \in [P_1]^{< \omega}$, $\phi_m : \langle \{\Lambda_{(\xi, n), c_m   \prod_{k < m} d(h_{\theta_k}^0(b_k))} : (\xi, n) \in G_{m}^{0}\} \rangle \to \mathbb{T}$ and $\psi_m : G_{m}^{1} \to \mathcal{B}$.
	
	\begin{claim}
		There exists $b_m \in R_{\theta_m}$, $r_{m + 1} > 0$, $s_{m+1}>0$ $G_{m + 1}^{0} \in [P_0 \times \omega]^{< \omega}$, $G_{m + 1}^{1} \in [P_1]^{< \omega}$, $\phi_{m + 1} : \langle \{\Lambda_{(\xi, n), c_{m + 1}   \prod_{k < m + 1} d(h_{\theta_k}^0(b_k))} : (\xi, n) \in G_{m + 1}^{0}\} \rangle \to \mathbb{T}$ and $\psi_{m + 1} : G_{m}^{1} \to \mathcal{B}$ satisfying the following conditions:
		
		\begin{enumerate}
			
			\item $b_m > b_{m - 1}$;

			\item $G_{m + 1}^{0} = G_{m}^{0} \cup \supp h_{\theta_m}^0(b_m) \cup \{e_{m + 1}^{0}\} \cup A_{m + 1}$;
			
			\item $G_{m + 1}^{1} = G_{m}^{1} \cup \supp h_{\theta_m}^1(b_m) \cup \{e_{m + 1}^{1}\} \cup B_{m + 1}$;
			
			\item $r_{m + 1} = \frac{r_m}{4   \sum_{\xi \in P_1} |a(h_{\theta_{m}}(b_m), \xi)|(\max\{((G_{m+1}^1\setminus G_{m}^1)\cap \omega) \cup \{0\}\}!)}$;
			
			\item $s_{m+1}=s_m d(h_{\theta_m}^1(b_m))  p_m$;
			
			\item $d(h_{\theta_m}^1(b_m))  p_m  \overline{\psi_{m + 1}(\xi)} \subset \psi_{m}(\xi)$ for every $\xi \in G_{m}^{1}$;
			
			\item $\delta(\psi_{m + 1}(\xi)) = r_{m + 1} / s_{m+1}$ for every $\xi \in G_{m + 1}^{1}$;
			
			\item $s_m  \psi_{m}(\theta_m) \cap (\phi_{m + 1}(h_{\theta_m}^0(b_m)) + \sum_{\xi \in P_1} a(h_{\theta_m}(b_m), \xi) s_m p_m \psi_{m + 1}(\xi)) \neq \emptyset$ for every $\theta_m \in J_1$;
			
			\item $\phi_{m + 1} \supset \phi_{m}$;
			
			\item $\phi_{m + 1}(h_{\theta_m}^0(b_m)) = \sum_{(\xi, n) \in \supp y_{\theta_{m}} \cap G_{m}^{0}} a(y_{\theta_{m}}, (\xi, n))  \phi_{m}(\Lambda_{(\xi, n), d((y_{\theta_{m}},0)})$, for every $\theta_m \in \bigcup_{n \in D} J_n$;
			
			\item if $n \in (G_{m + 1}^{1} \setminus G_{m}^{1}) \cap \omega$, then $p_m d(h_{\theta_m}^1(b_m))n! \psi_{m+1}(n)$ is cointained in the arc $(-r_m / s_m, r_m / s_m)+\mathbb Z$.
			
		\end{enumerate}
	\end{claim}
	
	\begin{proof}[Proof of the claim]
		Fix $b_m \in R_{\theta_{m}}$ such that $b_m > b_{m - 1}$ and
		
		\begin{itemize}
			
			\item $\supp h_{\theta_{m}}^{1, 1}(b_m) \setminus G_{m}^{1} \neq \emptyset$, if $h_{\theta_{m}}$ is of type 1;
			
			\item $q(h_{\theta_{m}}^{1, 1}(b_m)) r_m > 1$, if $h_{\theta_{m}}$ is of type 2;
			
			\item $|p(h_{\theta_{m}}(b_m), \mu)|  r_m > 4q(h_{\theta_{m}}(b_m), \mu)  p_m$ for some $\mu \in \supp h_{\theta_{m}}^{1, 1}(b_m)$, if $h_{\theta_{m}}$ is of type 3;
			
			\item $q(h_{\theta_{m}}^{1, 0}(b_m))  r_m > s_m$, if $h_{\theta_{m}}$ is of type 4;
			
			\item $|p(h_{\theta_{m}}(b_m), n)|  r_m > 8 n!  q(h_{\theta_{m}}(b_m), n)  p_m$ for some $n \in \supp h_{\theta_{m}}(b_m)$, if $h_{\theta_{m}}$ is of type 5;
			
			\item $n!  p_m  r_m > |h_{\theta_{m}}(b_m)(n)| s_m$ for some $n \in \supp h_{\theta_{m}}^{1, 0}(b_m) \setminus G_{m}^{1}$ if $h_{\theta_{m}}$ is of type 6;
			
			\item $\left| \frac{h_{\theta_m}(b_m)(n)}{n!}  a - t  a \right| < \frac{r_m  p_m}{32 s_m^{2}}$ for some $n \in \supp h_{\theta_{m}}^{1, 0}(b_m) \setminus G_{m}^{1}$, if $h_{\theta_{m}}$ is of type 7, where $a > l(r_m / (32 s_m), t)$ is previously fixed and $t \in \mathbb{R} \setminus \mathbb{Q}$ witness that $h_{\theta_{m}}$ is of type 7.
			
			\item $|h_{\theta_{m}}(b_m)(n)|  r_m > 8  p_m  n!$ for some $n \in \supp h_{\theta_{m}}^{1, 0}(b_m) \setminus G_{m}^{1}$, if $h_{\theta_{m}}$ is of type 8;
			
			\item $|\supp h_{\theta_{m}}^{1, 0}(b_m) \setminus G_{m}^{1}|  r_m  \min \left\{ \frac{h_{\theta_{m}}(n)(M)}{M!} : M \in\supp h_{\theta_{m}}(n), n \in \omega \right\} > 4$, if $h_{\theta_{m}}$ is of type 9;
			
			\item $n!  r_m > |h_{\theta_{m}}(b_m)(n)| s_m$ for some $n \in \supp h_{\theta_{m}}^{1, 0}(b_m) \setminus G_{m}^{1}$, if $h_{\theta_{m}}$ is of type 10;
			
			\item $q(h_{\theta_{m}}^0(b_m))  r_m >   c_{m + 1}  \prod_{k < m} d(h_{\theta_k}^0(b_k))$, if $h_{\theta_{m}}$ is of type 11;
			
			\item $\supp h_{\theta_{m}}(b_m) \setminus G_{m}^{0} \neq \emptyset$ and $\ordem(h_{\theta_{m}}(b_m) \upharpoonright \supp h_{\theta_{m}}(b_m) \setminus G_{m}^{0}) = \ordem(h_{\theta_{m}}(b_m))$, if $h_{\theta_{m}}$ is of type 12.
			
		\end{itemize}
	
		The type 12 case requires some explanation. Since $h_{\theta_m}$ is of type 12, there exists $k \in D$ such that $\{h_{\theta_m}(b): b \in \omega\}$ is a $k$-round subset of pairwise distict elements of $G$. The set $X=\{b>b_{m-1}: \supp h_{\theta_m}(b)\setminus G_m^0\neq \emptyset\}$ is infinite. Suppose by contradiction that for every $b \in X$, $\ordem(h_{\theta_{m}}(b) \upharpoonright \supp h_{\theta_{m}}(b) \setminus G_{m}^{0})$ is a proper divisor of $k$. Since the set of the proper divisors of  the number $k$ is finite, there exists a proper divisor $d$ of $k$ such that $X'=\{b \in X: \ordem(h_{\theta_{m}}(b) \upharpoonright \supp h_{\theta_{m}}(b) \setminus G_{m}^{0})=d\}$ is infinite. For each element $b$ of $X'$, $\supp(d  h_{\theta_m}(b))\subset G^0_m$. Since $k\{h_{\theta_m}(b): b \in X'\}=\{0\}$, it follows that $\varphi_d|_{\{h_{\theta_m}(b): b \in \omega\}}$ is not finite-to-one.
		
		If $\xi \in G_{m + 1}^{1} \setminus \supp h_{\theta_m}(b_m)$, let $\psi_{m + 1}(\xi)$ be an element of $\mathcal{B}$ satisfying (6), (7) and (11). If $h_{\theta_m}$ is of type $i \in \{1, \ldots, 10\}$, extend $\phi_{m}$ to $\phi_{m + 1}$ using the divisibility of $\mathbb{T}$. It remains to define $\phi_{m + 1}$ when $h_{\theta_m}$ is of type $i \in \{11, 12\}$ and $\psi_{m + 1}(\xi)$ for each $\xi \in \supp h_{\theta_m}^1(b_m)$.
		
		\begin{description}[style=unboxed,leftmargin=0cm]
			
			\item[Case 1] $h_{\theta_m}$ is of type 1. \\ Fix $\mu \in \supp h_{\theta_{m}}^{1, 1}(b_m) \setminus G_{m}^{1}$.
			
			If $\xi \in \supp h_{\theta_{m}}^1(b_m)  \cap G_{m}^{1}$, let $z_{\xi} \in \mathbb{T}$ be such that $d(h_{\theta_{m}}^1(b_m) )  p_m  z_{\xi}$ is the middle point of $\psi_{m}(\xi)$.
			
			If $\xi \in (\supp h_{\theta_{m}}^1(b_m)  \setminus G_{m}^{1}) \cap \omega$, set $z_{\xi} = 0 + \mathbb{Z}$.
			
			If $\xi \in [\supp h_{\theta_{m}}^1(b_m)   \setminus (G_{m}^{1} \cup \omega)] \setminus \{\mu\}$, let $z_{\xi}$ be any element of $\mathbb{T}$.
			
			Let $d = a(h_{\theta_{m}}(b_m), \mu)  s_m p_m$,  and $a$ be the middle point of $s_m\psi_{m}(\theta_{m})  - \sum_{\xi \in P_1 \setminus \{\mu\}} a(h_{\theta_m}(b_m), \xi)s_m p_m  z_{\xi}- \phi_{m + 1}(h_{\theta_m}^0(b_m)).$ Since $\mathbb T$ is divisible, let $z_{\mu} \in \mathbb{T}$ be such that $d  z_{\mu} = a$.
			
			For each $\xi \in \supp h_{\theta_{m}}^1(b_m) $, let $\psi_{m + 1}(\xi) \in \mathcal{B}$ be the open arc of $\mathbb{T}$ centered at $z_{\xi}$ satisfying (7).
			
			\item[Case 2] $h_{\theta_m}$ is of type 2. \\ Fix $\mu \in \supp h_{\theta_{m}}^{1, 1}(b_m)$ such that $q(h_{\theta_{m}}(b_m), \mu)  r_m > 1$. If $\mu\notin G^1_m$ we proceed as in the previous case. So suppose that $\mu \in G_{m}^{1}$.
			
			If $\xi \in (\supp h_{\theta_{m}}^1(b_m)  \cap G_{m}^{1}) \setminus \{\mu\}$, let $z_{\xi} \in \mathbb{T}$ be such that $d(h_{\theta_{m}}^1(b_m) )   p_m   z_{\xi}$ is the middle point of $\psi_{m}(\xi)$.
			
			If $\xi \in (\supp h_{\theta_{m}}^1(b_m)  \setminus G_{m}^{1}) \cap \omega$, set $z_{\xi} = 0 + \mathbb{Z}$.
			
			If $\xi \in \supp h_{\theta_{m}}^1(b_m)   \setminus (G_{m}^{1} \cup \omega)$, let $z_{\xi}$ be any element of $\mathbb{T}$.
			
			Apply Lemma \ref{lem2} to $c = p(h_{\theta_m}(b_m),\mu)$, $d = q(h_{\theta_{m}}(b_m), \mu)$, $\epsilon = r_m$, $a = \ \hbox{the middle point of} \ \psi_{m}(\mu)$ and $B = s_m \psi_{m}(\theta_{m}) - \phi_{m + 1}(h_{\theta_m}^0(b_m)) - \sum_{\xi \in P_1 \setminus \{\mu\}} a(h_{\theta_m}(b_m), \xi)  s_m p_m z_{\xi}$ in order to obtain $x \in \mathbb{T}$ such that $d  x = a$ and $c x \in B$. Since $\mathbb{T}$ is divisible, take $z_\mu\in\mathbb{T}$ such that $\frac{d(h_{\theta_m}(b_m))}{q(h_{\theta_{m}}(b_m)}s_mp_mz_\mu=x$.
			
			For each $\xi \in \supp h_{\theta_{m}}^1(b_m) $, let $\psi_{m + 1}(\xi) \in \mathcal{B}$ be the open arc of $\mathbb{T}$ centered at $z_{\xi}$ satisfying (7).
			
			\item[Case 3] $h_{\theta_m}$ is of type 3. \\ Fix $\mu \in \supp h_{\theta_{m}}^{1, 1}(b_m)$ such that $|p(h_{\theta_{m}}(b_m), \mu)|  r_m > q(h_{\theta_{m}}(b_m), \mu)  p_m$. We can suppose that $\mu \in G_{m}^{1}$, otherwise we proceed as in Case 1.
			
			If $\xi \in (\supp h_{\theta_{m}}^1(b_m)  \cap G_{m}^{1}) \setminus \{\mu\}$, let $z_{\xi} \in \mathbb{T}$ be such that $d(h_{\theta_{m}}^1(b_m) )  p_m  z_{\xi}$ is the middle point of $\psi_{m}(\xi)$.
			
			If $\xi \in (\supp h_{\theta_{m}}^1(b_m)  \setminus G_{m}^{1}) \cap \omega$, set $z_{\xi} = 0 + \mathbb{Z}$.
			
			If $\xi \in \supp h_{\theta_{m}}^1(b_m)   \setminus (G_{m}^{1} \cup \omega)$, let $z_{\xi}$ be any element of $\mathbb{T}$.
			
			Apply Lemma \ref{lem2} to $d = a(h_{\theta_{m}}(b_m), \mu) s_m p_m$, $c = d(h_{\theta_m}^1(b_m))  p_m$, $\epsilon = r_m / (4s_m)$, $B =$ is the arc of same center as $\psi_{m}(\mu)$ and diameter $\epsilon$ and $a = \ \hbox{the middle point of} \ s_m  \psi_{m}(\theta_{m}) - \phi_{m + 1}(h_{\theta_m}^0(b_m)) - \sum_{\xi \in P_1 \setminus \{\mu\}} a(h_{\theta_m}(b_m), \xi)  s_m p_m z_{\xi}$ in order to obtain $z_{\mu} \in \mathbb{T}$ such that $c  z_{\mu} \in B$ and $d  z_{\mu} = a$.
			
			For each $\xi \in \supp h_{\theta_{m}}^1(b_m) $, let $\psi_{m + 1}(\xi) \in \mathcal{B}$ be the open arc of $\mathbb{T}$ centered at $z_{\xi}$ satisfying (7).
			
			\item[Case 4] $h_{\theta_m}$ is of type 4. \\ Fix $n \in \supp h_{\theta_{m}}^{1, 0}(b_m)$ such that $q(h_{\theta_{m}}(b_m), n)  r_m > s_m$. If $n \in G^1_m$, we proceed as in Case 2. If $n \notin G^1_m$, we proceed as follows:
			
			If $\xi \in (\supp h_{\theta_{m}}^1(b_m)  \cap G_{m}^{1})$, let $z_{\xi} \in \mathbb{T}$ be such that $d(h_{\theta_{m}}^1(b_m) )  p_m  z_{\xi}$ is the middle point of $\psi_{m}(\xi)$.
			
			If $\xi \in [(\supp h_{\theta_{m}}^1(b_m)  \setminus G_{m}^{1}) \cap \omega] \setminus \{n\}$, set $z_{\xi} = 0 + \mathbb{Z}$.
			
			If $\xi \in \supp h_{\theta_{m}}^1(b_m)  \setminus (G_{m}^{1} \cup \omega)$, let $z_{\xi}$ be any element of $\mathbb{T}$.
			
	
			Apply Lemma \ref{lem2} to $c = p(h_{\theta_m}(b_m),n)s_m$, $d =q(h_{\theta_m}(b_m),n)$, $\epsilon = r_m$, $a = 0 + \mathbb{Z}$ and $B = s_m  \psi_{m}(\theta_{m}) - \phi_{m + 1}(h_{\theta_m}^0(b_m)) - \sum_{\xi \in P_1 \setminus \{n\}} a(h_{\theta_m}(b_m), \xi)  s_mp_m z_{\xi}$ in order to obtain $x \in \mathbb{T}$ such that $d  x = a$ and $c x \in B$. Since $\mathbb{T}$ is divisible, take $z_n\in\mathbb{T}$ such that $\frac{d(h_{\theta_m}(b_m))}{q(h_{\theta_{m}}(b_m)}p_mz_n=x$.

			For each $\xi \in \supp h_{\theta_{m}}^1(b_m) $, let $\psi_{m + 1}(\xi) \in \mathcal{B}$ be the open arc of $\mathbb{T}$ centered at $z_{\xi}$ satisfying (7).
			
			\item[Case 5] $h_{\theta_m}$ is of type 5. \\ Fix $n \in \supp h_{\theta_{m}}^{1, 0}(b_m)$ such that $|p(h_{\theta_{m}}(b_m), n)|   r_m > 8 n!   q(h_{\theta_{m}}(b_m), n)   p_m$. If $n \in G^1_m$ we proceed as in Case 3. So suppose $n \notin G^1_m$.
			
			If $\xi \in (\supp h_{\theta_{m}}^1(b_m)  \cap G_{m}^{1})$, let $z_{\xi} \in \mathbb{T}$ be such that $d(h_{\theta_{m}}^1(b_m) )   p_m   z_{\xi}$ is the middle point of $\psi_{m}(\xi)$.
			
			If $\xi \in [(\supp h_{\theta_{m}}^1(b_m)  \setminus G_{m}^{1}) \cap \omega] \setminus \{n\}$, set $z_{\xi} = 0 + \mathbb{Z}$.
			
			If $\xi \in \supp h_{\theta_{m}}^1(b_m)  \setminus (G_{m}^{1} \cup \omega)$, let $z_{\xi}$ be any element of $\mathbb{T}$.
			
			Apply Lemma \ref{lem2} to $d = a(h_{\theta_{m}}(b_m), n)  p_ms_m$, $c = d(h_{\theta_m}^1(b_m)  p_m   n!$, $\epsilon = r_m / (8  s_m)$, $B$ is the open arc of $\mathbb{T}$ centered at $0 + \mathbb{Z}$ with length $\epsilon$ and $a$ is the middle point of $s_m \psi_{m}(\theta_{m}) - \phi_{m + 1}(h_{\theta_m}^0(b_m)) - \sum_{\xi \in P_1 \setminus \{n\}} a(h_{\theta_m}(b_m), \xi) p_ms_m z_{\xi}$ in order to obtain $z_{n} \in \mathbb{T}$ such that $c   z_{n} \in B$ and $d   z_{n} = a$.
			
			For each $\xi \in \supp h_{\theta_{m}}^1(b_m) $, let $\psi_{m + 1}(\xi) \in \mathcal{B}$ be the open arc of $\mathbb{T}$ centered at $z_{\xi}$ satisfying (7).
			
			\item[Case 6] $h_{\theta_m}$ is of type 6. \\
			 Fix $n \in \supp h_{\theta_{m}}^{1, 0}(b_m) \setminus G_{m}^{1}$ such that $n!   p_m   r_m > |h_{\theta_{m}}(b_m)(n)| s_m$.
			
			If $\xi \in \supp h_{\theta_{m}}^1(b_m)  \cap G_{m}^{1}$, let $z_{\xi} \in \mathbb{T}$ be such that $d(h_{\theta_{m}}^1(b_m) )   p_m   z_{\xi}$ is the middle point of $\psi_{m}(\xi)$.
			
			If $\xi \in [(\supp h_{\theta_{m}}^1(b_m)  \setminus G_{m}^{1}) \cap \omega] \setminus \{n\}$, set $z_{\xi} = 0 + \mathbb{Z}$.
			
			If $\xi \in \supp h_{\theta_{m}}^1(b_m)  \setminus (G_{m}^{1} \cup \omega)$, let $z_{\xi}$ be any element of $\mathbb{T}$.
			
			Apply Lemma \ref{lem2} to $d = d(h_{\theta_m}^1(b_m)   p_m   n!$, $c = a(h_{\theta_{m}}(b_m), n)  s_m$, $\epsilon = r_m$, $B =s_m  \psi_{m}(\theta_{m}) - \phi_{m + 1}(h_{\theta_m}^0(b_m)) - \sum_{\xi \in P_1 \setminus \{n\}} a(h_{\theta_m}(b_m), \xi)   s_m p_m   z_{\xi}$ and $a = 0 + \mathbb{Z}$ in order to obtain $z_{n} \in \mathbb{T}$ such that $c   z_{n} \in B$ and $d   z_{n} = a$.
			
			For each $\xi \in \supp h_{\theta_{m}}^1(b_m) $, let $\psi_{m + 1}(\xi) \in \mathcal{B}$ be the open arc of $\mathbb{T}$ centered at $z_{\xi}$ satisfying (7).
			
			\item[Case 7] $h_{\theta_m}$ is of type 7. \\ Fix $a > l(r_m / (32   s_m), t)$, where $t \in \mathbb{R} \setminus \mathbb{Q}$ witness that $h_{\theta_{m}}$ is of type 7. Fix $n \in \supp h_{\theta_{m}}^{1, 0}(b_m) \setminus G_{m}^{1}$ such that $\left| \frac{h_{\theta_m}(b_m)(n)}{n!}   a - t   a \right| < \frac{r_m   p_m}{32   s_m^{2}}$.
			
			If $\xi \in \supp h_{\theta_{m}}^1(b_m)  \cap G_{m}^{1}$, let $z_{\xi} \in \mathbb{T}$ be such that $d(h_{\theta_{m}}^1(b_m) )   p_m   z_{\xi}$ is the middle point of $\psi_{m}(\xi)$.
			
			If $\xi \in [(\supp h_{\theta_{m}}^1(b_m)  \setminus G_{m}^{1}) \cap \omega] \setminus \{n\}$, set $z_{\xi} = 0 + \mathbb{Z}$.
			
			If $\xi \in \supp h_{\theta_{m}}^1(b_m)  \setminus (G_{m}^{1} \cup \omega)$, let $z_{\xi}$ be any element of $\mathbb{T}$.
			
			Apply Lemma \ref{lem_kronecker_T2} to $c_1 = a(h_{\theta_{m}}(b_m), n)   p_m$, $c_2 = d(h_{\theta_m}^1(b_m)   p_m   n!$, $\epsilon = r_m / (8   s_m^2)$, $\tilde B_1 = s_m   \psi_{m}(\theta_{m}) - \phi_{m + 1}(h_{\theta_m}^0(b_m)) - \sum_{\xi \in P_1 \setminus \{n\}} a(h_{\theta_m}(b_m), \xi)   s_m p_m   z_{\xi}$, $B_1$ is an open arc such that $s_mB_1=\tilde B_1$ and $B_2$ is the open arc of $\mathbb{T}$ centered at $0 + \mathbb{Z}$ with length $\epsilon$ in order to obtain $z_n \in \mathbb{T}$ such that $s_mc_1   z_n \in \tilde B_1$ and $c_2   z_n \in B_2$.
			
			For each $\xi \in \supp h_{\theta_{m}}^1(b_m) $, let $\psi_{m + 1}(\xi) \in \mathcal{B}$ be the open arc of $\mathbb{T}$ centered at $z_{\xi}$ satisfying (7).
			
			\item[Case 8] $h_{\theta_m}$ is of type 8. \\ Fix $n \in \supp h_{\theta_{m}}^{1, 0}(b_m) \setminus G_{m}^{1}$ such that $|h_{\theta_{m}}(b_m)(n)|   r_m > 8   p_m   n!$.
			
			If $\xi \in \supp h_{\theta_{m}}^1(b_m)  \cap G_{m}^{1}$, let $z_{\xi} \in \mathbb{T}$ be such that $d(h_{\theta_{m}}^1(b_m) )   p_m   z_{\xi}$ is the middle point of $\psi_{m}(\xi)$.
			
			If $\xi \in [(\supp h_{\theta_{m}}^1(b_m)  \setminus G_{m}^{1}) \cap \omega] \setminus \{n\}$, set $z_{\xi} = 0 + \mathbb{Z}$.
			
			If $\xi \in \supp h_{\theta_{m}}^1(b_m)  \setminus (G_{m}^{1} \cup \omega)$, let $z_{\xi}$ be any element of $\mathbb{T}$.
			
			Apply Lemma \ref{lem2} to $d = a(h_{\theta_{m}}(b_m), n)   s_mp_m$, $c = d(h_{\theta_m}^1(b_m)   p_m   n!$, $\epsilon = r_m / (8   s_m)$, $B$ is the the open arc of $\mathbb{T}$ centered at $0 + \mathbb{Z}$ with length $\epsilon$ and $a$ is the middle point of $s_m   \psi_{m}(\theta_{m}) - \phi_{m + 1}(h_{\theta_m}^0(b_m)) - \sum_{\xi \in P_1 \setminus \{n\}} a(h_{\theta_m}(b_m), \xi)   s_mp_m   z_{\xi}$ in order to obtain $z_{n} \in \mathbb{T}$ such that $c   z_{n} \in B$ and $d   z_{n} = a$.
			
			For each $\xi \in \supp h_{\theta_{m}}^1(b_m) $, let $\psi_{m + 1}(\xi) \in \mathcal{B}$ be the open arc of $\mathbb{T}$ centered at $z_{\xi}$ satisfying (7).
			
			\item[Case 9] $h_{\theta_m}$ is of type 9. \\
			If $\xi \in \supp h_{\theta_{m}}^1(b_m)  \cap G_{m}^{1}$, let $z_{\xi} \in \mathbb{T}$ be such that $d(h_{\theta_{m}}^1(b_m) )   p_m   z_{\xi}$ is the middle point of $\psi_{m}(\xi)$.
			
			If $\xi \in \supp h_{\theta_{m}}^1(b_m)  \setminus (G_{m}^{1} \cup \omega)$, let $z_{\xi}$ be any element of $\mathbb{T}$.
			
			If $n \in (\supp h_{\theta_{m}}^1(b_m)  \setminus G_{m}^{1}) \cap \omega$, let $A_n \in \mathcal{B}$ be centered at $0 + \mathbb{Z}$ with length $r_m / (4   s_{m+1}   n!)$.
			
			Since $\sum_{n \in \omega \setminus G_{m}^{1}} |a(h_{\theta_{m}}(b_m), n)|   s_mp_m  \delta(A_n) = \sum_{n \in \omega \setminus G_{m}^{1}} \frac{|h_{\theta_{m}}(b_m)(n)|}{n!}   \frac{r_m}{4} \geq |\supp h_{\theta_{m}}^{1, 0}(b_m) \setminus G_{m}^{1}|   \min \left\{ \frac{|h_{\theta_{m}}(b_m)(n)|}{n!} : n \in \supp h_{\theta_{m}}(b_m) \right\}   \frac{r_m}{4} > 1$, it follows that $\sum_{n \in \omega \setminus G_{m}^{1}} a(h_{\theta_{m}}(b_m), n)   s_m   A_n = \mathbb{T}$.
			
		There exists $z_n \in A_n$ such that $\sum_{\xi \in \omega \setminus G_{m}^{1}} a(h_{\theta_m}(b_m), n)   s_m   z_{n} \in s_m   \psi_{m}(\theta_{m}) - \phi_{m + 1}(h_{\theta_m}^0(b_m)) - \sum_{\xi \in P_1 \setminus (\omega \setminus G_{m}^{1})} a(h_{\theta_m}(b_m), \xi)   p_ms_m   z_{\xi}$ for each $n \in \supp h_{\theta_{m}}^{1, 0}(b_m)  \setminus G_{m}^{1}$.
			
			For each $\xi \in \supp h_{\theta_{m}}^1(b_m) $, let $\psi_{m + 1}(\xi) \in \mathcal{B}$ be the open arc of $\mathbb{T}$ centered at $z_{\xi}$ satisfying (7).
			
			\item[Case 10] $h_{\theta_m}$ is of type 10. \\ Fix $n \in \supp h_{\theta_{m}}^{1, 0}(b_m) \setminus G_{m}^{1}$ such that $n!     r_m > |h_{\theta_{m}}(b_m)(n)|   s_m$.
			
			If $\xi \in \supp h_{\theta_{m}}^1(b_m)  \cap G_{m}^{1}$, let $z_{\xi} \in \mathbb{T}$ be such that $d(h_{\theta_{m}}^1(b_m) )   p_m   z_{\xi}$ is the middle point of $\psi_{m}(\xi)$.
			
			If $\xi \in [(\supp h_{\theta_{m}}^1(b_m)  \setminus G_{m}^{1}) \cap \omega] \setminus \{n\}$, set $z_{\xi} = 0 + \mathbb{Z}$.
			
			If $\xi \in \supp h_{\theta_{m}}^1(b_m)  \setminus (G_{m}^{1} \cup \omega)$, let $z_{\xi}$ be any element of $\mathbb{T}$.
			
			Apply Lemma \ref{lem2} to $d = d(h_{\theta_m}^1(b_m))   p_m   n!$, $c = a(h_{\theta_{m}}(b_m), n)   s_mp_m$, $\epsilon = r_m$, $B = s_m   \psi_{m}(\theta_{m}) - \phi_{m + 1}(h_{\theta_m}^0(b_m)) - \sum_{\xi \in P_1 \setminus \{n\}} a(h_{\theta_m}(b_m), \xi)   s_mp_m   z_{\xi}$ and $a = 0 + \mathbb{Z}$ in order to obtain $z_{n} \in \mathbb{T}$ such that $c   z_{n} \in B$ and $d   z_{n} = a$.
			
			For each $\xi \in \supp h_{\theta_{m}}^1(b_m) $, let $\psi_{m + 1}(\xi) \in \mathcal{B}$ be the open arc of $\mathbb{T}$ centered at $z_{\xi}$ satisfying (7).
			
			\item[Case 11] $h_{\theta_m}$ is of type 11. \\ If $\xi \in \supp h_{\theta_m}^1(b_m)$, let $\psi_{m + 1}(\xi)$ be an element of $\mathcal{B}$ satisfying (6), (7) and (11). Let $v_m=c_{m+1}   \prod_{k < m} d(h_{\theta_k}^0(b_k))$. Use the divisibility of $\mathbb{T}$ to extend $\phi_{m}$ to a group homomorphism $\tilde{\phi}_{m + 1} : \langle \{\Lambda_{(\xi, n), v_m} : (\xi, n) \in G_{m + 1}^{0}\} \rangle \to \mathbb{T}$.
			
			Fix $(\mu, l) \in \supp h_{\theta_m}^0(b_m)$ such that $q(h_{\theta_m}(b_m), (\mu, l))   r_m > v_m$.	If $(\xi, n) \in G_{m + 1}^{0} \setminus \{(\mu, l)\}$, let $x_{(\xi, n)} \in \mathbb{T}$ be such that $d(h_{\theta_m}^0(b_m))   x_{(\xi, n)} = \tilde{\phi}_{m + 1}(\Lambda_{(\xi, n), v_m})$.
			
			Apply Lemma \ref{lem2} by using $c = p(h_{\theta_m}^0(b_m)(\mu))$, $d = q(h_{\theta_m}^0(b_m)(\mu))$, $\epsilon = r_m/v_m$, $a = \tilde{\phi}_{m + 1}(\Lambda_{(\mu, l), v_m})$ $\tilde B = s_m   \psi_{m}(\theta_{m}) - \sum_{(\xi, n) \in (P_0 \times \omega) \setminus \{(\mu, l)\}} a(h_{\theta_m}^0(b_m), (\xi, n))   v_m   x_{(\xi, n)} - \sum_{\xi \in P_1} a(h_{\theta_m}(b_m), \xi)   s_mp_m   z_{\xi}$ and $B$ be an open arc such that $v_mB=\tilde B$ where $z_{\xi}$ is the middle point of $\psi_{m + 1}(\xi)$ for every $\xi \in \supp h_{\theta_m}^1(b_m)$, in order to obtain $x \in \mathbb{T}$ such that $dx = a$ and $cx \in B$. Let $x_{(\mu, l)} \in \mathbb T$ be such that $\frac{d(h^0_{\theta_m}(b_m)(\mu))}{q(h^0_{\theta_m}(b_m)(\mu))}x_{(\mu, l)}=x$.
			
			Extend $\tilde{\phi}_{m + 1}$ to a group homomorphism $\phi_{m + 1} : \langle \{\Lambda_{(\xi, n),    v_md(h_{\theta_m}^0(b_m))} : (\xi, n) \in G_{m + 1}^{0}\} \rangle \to \mathbb{T}$ such that $\phi_{m + 1}(\Lambda_{(\xi, n), v_md(h_{\theta_m}^0(b_m))}) = x_{(\xi, n)}$ for every $(\xi, n) \in G_{m + 1}^{0}$.
			
			\item[Case 12] $h_{\theta_m}$ is of type 12. \\ In this case, $\supp h_{\theta_m}^1(b_m) = \emptyset$. Note that $\{h_{\theta_{m}}^0(b_m)\} \cup \{\Lambda_{(\xi, n), c_m   \prod_{k < m} d(h_{\theta_{k}}^0(b_k) )} : (\xi, n) \in G_{m}^{0}\}$ is an independent subset of the group $(Q / Z)^{((P_0 \cap E) \times \omega)}$. Since $\ordem(h_{\theta_{m}}(b_m)) = \ordem(y_{\theta_{m}})$, we can extend $\phi_{m}$ to a group homomorphism $\phi_{m + 1} : \langle \{\Lambda_{(\xi, n), c_{m + 1}   \prod_{k < m + 1} d(h_{\theta_k}^0(b_k))} : (\xi, n) \in G_{m + 1}^{0}\} \rangle \to \mathbb{T}$ satisfying (10). \qedhere
			
		\end{description}
	\end{proof}
	
	By induction, we have $r_{m + 1} > 0$, $b_m \in R_{\theta_m}$, $G_{m + 1}^{0} \in [P_0 \times \omega]^{< \omega}$, $G_{m + 1}^{1} \in [P_1]^{< \omega}$, $\phi_{m + 1} : \langle \{\Lambda_{(\mu, l), c_{m + 1}   \prod_{k < m + 1} d(h_{\theta_k}^0(b_k))} : (\mu, l) \in G_{m + 1}^{0}\} \rangle \to \mathbb{T}$ and $\psi_{m + 1} : G_{m}^{1} \to \mathcal{B}$ satisfying (1)-(11) for every $m \in \omega$.
	
	Since $\mathbb{T}$ is a divisible group, it is possible to extend $\bigcup_{m \in \omega} \phi_{m}$ to a group homomorphism $\phi_{\alpha, E}^{0} : (\mathbb{Q} / \mathbb{Z})^{((P_0 \cap E) \times \omega)} \to \mathbb{T}$.
	
	Since $\mathbb{T}$ is a complete metric space and $(r_m : m \in \omega)$ is a sequence of positive real numbers converging to 0, we conclude from (1), (6) and (7) that if $\xi \in P_1 \cap E$, then $\bigcap_{m \geq N_{\xi}} s_m  \psi_{m}(\xi) = \bigcap_{m \geq N_{\xi}} s_m   \overline{\psi_{m}(\xi)}$ is a one-element set, where $N_{\xi} = \min\{m \in \omega : \xi \in G_{m}^{1}\}$. Denote by $\phi_{\alpha, E}^{1}(\chi_{\xi})$ the unique element of this set. If $m \geq N_{\xi}$, then there exists a unique element of $\psi_{m}(\xi)$ whose multiplication by $s_m$ is equal to $\phi_{\alpha, E}^{1}(\chi_{\xi})$. Denote this element by $\phi_{\alpha, E}^{1} \left( \frac{1}{s_m}   \chi_{\xi} \right)$. The divisibility of $\mathbb{T}$ allows us to extend $\phi_{\alpha, E}^{1}$ to a group homomorphism whose domain is $\mathbb{Q}^{(P_1 \cap E)}$ and whose range is $\mathbb{T}$.
	
	Define $\phi_{\alpha, E} : (\mathbb{Q} / \mathbb{Z})^{((P_0 \cap E) \times \omega)} \oplus \mathbb{Q}^{(P_1 \cap E)} \to \mathbb{T}$ by $\phi_{\alpha, E}(H_0, H_1) = \phi_{\alpha, E}^{0}(H_0) + \phi_{\alpha, E}^{1}(H_1)$ for every $(H_0, H_1) \in (\mathbb{Q} / \mathbb{Z})^{((P_0 \cap E) \times \omega)} \oplus \mathbb{Q}^{(P_1 \cap E)}$. It remains to show that (i)-(iv) are satisfied.
	
	If $H_{1} = 0$, then $\phi_{\alpha, E}(H_\alpha) = \phi_{\alpha, E}^{0}(H_{0}) = \phi_{0}(H_{0}) \neq 0 + \mathbb{Z}$. If $H_{1} \neq 0$, then $\phi_{\alpha, E}(H_\alpha) \in \phi_{0}(H_{0}) + \sum_{\xi \in P_1} a(H_{\alpha}, \xi)   \psi_{0}(\xi)$. Since $0 + \mathbb{Z} \not \in \phi_{0}(H_{0}) + \sum_{\xi \in P_1} a(H_{\alpha}, \xi)   \psi_{0}(\xi)$, (i) holds.
	
	If $\xi \in E$, set $I_{\xi} = \{m \in \omega : \xi = \theta_{m}\}$ and $S_{\xi} = \{b_m : m \in I_{\xi}\}$. It is clear that $S_{\xi} \in [R_{\xi}]^{\omega}$. We will show that if $\xi \in E \cap J_1$ (respectively, $\xi \in E \cap \bigcup_{n \in D} J_n$), then the sequence $(\phi_{\alpha, E}(h_{\xi}(n)) : n \in S_{\xi})$ converges to $\phi_{\alpha, E}(0, \chi_{\xi}) = \phi_{\alpha, E}^{1}(\chi_{\xi})$ (respectively, $\phi_{\alpha, E}(y_{\xi}, 0) = \phi_{\alpha, E}^{0}(y_{\xi})$).
	
	Fix $\xi \in E \cap J_1$. Since $\phi_{\alpha, E}(h_{\theta_{m}}(b_m))$ is an element of the arc $ \phi_{m + 1}(h_{\theta_{m}}^0(b_m)) + \sum_{\xi \in P_1} a(h_{\theta_{m}}(b_m), \xi)   s_mp_m   \psi_{m + 1}(\xi)$ and $\phi_{\alpha, E}(0, \chi_{\theta_{m}}) \in s_m  \psi_{m + 1}(\theta_{m})$, it follows from (7) and (8) that $\delta(\phi_{\alpha, E}(h_{\theta_{m}}(b_m)), \phi_{\alpha, E}(0, \chi_{\theta_{m}})) \leq \sum_{\xi \in P_1} |a(h_{\theta_{m}}(b_m), \xi)|   \frac{r_{m + 1}}{d(h_{\theta_{m}}^1(b_m) )} + r_m < 2   r_m$. It follows from $r_m \to 0$ that (ii) holds.
	
	It follows from (10) that if $\xi \in E \cap \bigcup_{n \in D} J_n$, then the sequence $(\phi_{\alpha, E}(h_{\xi}(n)) : n \in S_{\xi})$ converges to $\phi_{\alpha, E}^{0}(y_{\xi}) = \phi_{\alpha, E}(y_{\xi}, 0)$. So, (iii) holds.
	
	Fix $p \in \mathbb{Z} \setminus \{0\}$.  Let $\epsilon>0$ be given. There exists $j$ such that $p|s_j$ and $r_j<\epsilon$. Let $n \in \omega\setminus G^1_j$ and let $m\geq j$ be the first number such that $n \in G_{m+1}^1$. By definition, $\phi_{\alpha, E} \left( 0, \frac{1}{s_{m+1}}   \chi_{n} \right)\in\psi_{m+1}(n)$, therefore, by (11), $n!\phi_{\alpha, E} \left( 0, \frac{1}{s_{m}}\right)\subset (-\frac{r_m}{s_m}, \frac{r_m}{s_m})+\mathbb Z$. Then, multiplying by $\frac{s_m}{p}$, $n!\phi_{\alpha, E} \left( 0, \frac{1}{s_{m}}\right)\subset (-\frac{r_m}{|p|}, \frac{r_m}{|p|})+\mathbb Z\subset (-r_m, r_m)+\mathbb Z$, so (4) follows.

\end{proof}

\begin{lem}\label{lem_ext_hom}
	Let $\alpha < \mathfrak{c}$. For each $\xi \in J_1 \cup \bigcup_{n \in D} J_n$, consider $R_{\xi} \in [\omega]^{\omega}$. There exists a group homomorphism $\phi_{\alpha}: (\mathbb{Q} / \mathbb{Z})^{(P_0 \times \omega)} \oplus \mathbb{Q}^{(P_1)} \to \mathbb{T}$ satisfying the following conditions:
	
	\begin{enumerate}[label=(\roman*)]
		
		\item $\phi_{\alpha}(H_{\alpha}) \neq 0 + \mathbb{Z}$;
		
		\item for each $\xi \in J_1$, there exists $S_{\xi} \in [R_{\xi}]^{\omega}$ such that the sequence $(\phi_{\alpha}(h_{\xi}(n)) : n \in S_{\xi})$ converges to $\phi_{\alpha}(0, \chi_{\xi})$;
		
		\item for each $\xi \in \bigcup_{n \in D} J_n$, there exists $S_{\xi} \in [R_{\xi}]^{\omega}$ such that the sequence $(\phi_{\alpha}(h_{\xi}(n)) : n \in S_{\xi})$ converges to $\phi_{\alpha}(y_{\xi}, 0)$;
		
		\item for each $p \in \mathbb{Z} \setminus \{0\}$, the sequence $\left( \phi_{\alpha} \left( 0, \frac{1}{p}   n!   \chi_{n} \right) : n \in \omega \right)$ converges to $0 + \mathbb{Z}$.
		
	\end{enumerate}
\end{lem}

\begin{proof}
	According to Proposition \ref{prop_const_E}, there exists $E \in [\mathfrak{c}]^{\omega}$ such that $\supp H_{\alpha} \subset (E \times \omega) \cup E$, $|E \cap (J_1 \cup \bigcup_{n \in D} J_n)| = \omega$, $\omega \subset E$ and $\bigcup_{n \in \omega} \supp h_{\xi}(n) \subset (E \times \omega) \cup E$ for every $\xi \in E \cap (J_1 \cup \bigcup_{n \in D} J_n)$.
	
	It follows from Lemma \ref{lem_const_hom} that there exists a group homomorphism $\phi_{\alpha, E} : (\mathbb{Q} / \mathbb{Z})^{((P_0 \cap E) \times \omega)} \oplus \mathbb{Q}^{(P_1 \cap E)} \to \mathbb{T}$ satisfying the following conditions:
	
	\begin{enumerate}[label=(\arabic*)]
		
		\item $\phi_{\alpha, E}(H_{\alpha}) \neq 0 + \mathbb{Z}$;
		
		\item for each $\xi \in E \cap J_1$, there exists $S_{\xi} \in [R_{\xi}]^{\omega}$ such that the sequence $(\phi_{\alpha, E}(h_{\xi}(n)) : n \in S_{\xi})$ converges to $\phi_{\alpha, E}(0, \chi_{\xi})$;
		
		\item for each $\xi \in E \cap \bigcup_{n \in D} J_n$, there exists $S_{\xi} \in [R_{\xi}]^{\omega}$ such that the sequence $(\phi_{\alpha, E}(h_{\xi}(n)) : n \in S_{\xi})$ converges to $\phi_{\alpha, E}(y_{\xi}, 0)$;
		
		\item for each $p \in \mathbb{Z} \setminus \{0\}$, the sequence $\left( \phi_{\alpha, E} \left( 0, \frac{1}{p}   n!   \chi_{n} \right) : n \in \omega \right)$ converges to $0 + \mathbb{Z}$.
		
	\end{enumerate}
	
	Let $\{\alpha_{\xi} : \xi < \mathfrak{c}\}$ be a strictly increasing enumeration of $\mathfrak{c} \setminus E$. Choose $S_{\alpha_{0}} \in [R_{\alpha_{0}}]^{\omega}$ such that $\{\phi_{\alpha, E}(h_{\alpha_{0}}(n)) : n \in S_{\alpha_{0}}\}$ is convergent. Note that this is possible, since $\mathbb{T}$ is sequentially compact.
	
	If $\alpha_{0} \in J_1$, denote by $\tilde{\phi}_{\alpha, E \cup \{\alpha_{0}\}}(0, \chi_{\alpha_{0}})$ the limit point of $(\phi_{\alpha, E}(h_{\alpha_{0}}(n)) : n \in S_{\alpha_{0}})$ and put $\tilde{\phi}_{\alpha, E \cup \{\alpha_{0}\}}(H_0, H_1) = \phi_{\alpha, E}(H_0, H_1)$ for every $(H_0, H_1) \in (\mathbb{Q} / \mathbb{Z})^{((P_0 \cap E) \times \omega)} \oplus \mathbb{Q}^{(P_1 \cap E)}$. Since $\mathbb{T}$ is divisible, it is possible to extend $\tilde{\phi}_{\alpha, E \cup \{\alpha_{0}\}}$ to a group homomorphism $\phi_{\alpha, E \cup \{\alpha_{0}\}} : (\mathbb{Q} / \mathbb{Z})^{((P_0 \cap E) \times \omega)} \oplus \mathbb{Q}^{((P_1 \cap E) \cup \{\alpha_{0}\})} \to \mathbb{T}$.
	
	If $\alpha_{0} \in \cup_{n \in D} J_n$, denote by $\tilde{\phi}_{\alpha, E \cup \{\alpha_{0}\}}(y_{\alpha_{0}}, 0)$ the limit point of $(\phi_{\alpha, E}(h_{\alpha_{0}}(n)) : n \in S_{\alpha_{0}})$ and put $\tilde{\phi}_{\alpha, E \cup \{\alpha_{0}\}}(H_0, H_1) = \phi_{\alpha, E}(H_0, H_1)$ for every $(H_0, H_1) \in (\mathbb{Q} / \mathbb{Z})^{((P_0 \cap E) \times \omega)} \oplus \mathbb{Q}^{(P_1 \cap E)}$. Since $\mathbb{T}$ is divisible, it is possible to extend $\tilde{\phi}_{\alpha, E \cup \{\alpha_{0}\}}$ to a group homomorphism $\phi_{\alpha, E \cup \{\alpha_{0}\}} : (\mathbb{Q} / \mathbb{Z})^{(((P_0 \cap E) \cup \{\alpha_{0}\}) \times \omega)} \oplus \mathbb{Q}^{(P_1 \cap E)} \to \mathbb{T}$.
	
	Repeating this construction inductively, we obtain a group homomorphism $\phi_{\alpha}: (\mathbb{Q} / \mathbb{Z})^{(P_0 \times \omega)} \oplus \mathbb{Q}^{(P_1)} \to \mathbb{T}$ satisfying (i)-(iv).
\end{proof}

The assumption $\mathfrak{p} = \mathfrak{c}$ together with Lemma \ref{lem_ext_hom} implies Proposition \ref{prop_hom}, which will be restated and proved below.

\begin{myprop2}[$\mathfrak{p} = \mathfrak{c}$]
	For each $\alpha < \mathfrak{c}$ and each $\xi \in J_1 \cup \bigcup_{n \in D} J_n$ there exists $S_{\xi, \alpha} \in [\omega]^{\omega}$ such that if $\alpha < \beta < \mathfrak{c}$, then $S_{\xi, \beta} \subset^{*} S_{\xi, \alpha}$. There also exists a group homomorphism $\phi_{\alpha} : (\mathbb{Q} / \mathbb{Z})^{(P_0 \times \omega)} \oplus \mathbb{Q}^{(P_1)} \to \mathbb{T}$ satisfying the following conditions:
	
	\begin{enumerate}[label=(\roman*)]
		
		\item $\phi_{\alpha}(H_{\alpha}) \neq 0 + \mathbb{Z}$;
		
		\item if $\xi \in J_1$, then the sequence $(\phi_{\alpha}(h_{\xi}(n)) : n \in S_{\xi, \alpha})$ converges to $\phi_{\alpha}(0, \chi_{\xi})$;
		
		\item if $\xi \in \bigcup_{n \in D} J_n$, then the sequence $(\phi_{\alpha}(h_{\xi}(n)) : n \in S_{\xi, \alpha})$ converges to $\phi_{\alpha}(y_{\xi}, 0)$;
		
		\item for each $p \in \mathbb{Z} \setminus \{0\}$, the sequence $\left(\phi_{\alpha} \left( 0, \frac{1}{p}   n!   \chi_{n} \right) : n \in \omega \right)$ converges to $0 + \mathbb{Z}$.
		
	\end{enumerate}
\end{myprop2}

\begin{proof}
	For each $\xi \in J_1 \cup \bigcup_{n \in D} J_n$, put $R_{\xi, 0} = \omega$. Applying Lemma \ref{lem_ext_hom} to $\alpha = 0$ and $R_{\xi} = R_{\xi, 0}$, we obtain $S_{\xi, 0} \in [R_{\xi, 0}]^{\omega}$ and a group homomorphism $\phi_{0}: (\mathbb{Q} / \mathbb{Z})^{(P_0 \times \omega)} \oplus \mathbb{Q}^{(P_1)} \to \mathbb{T}$ satisfying (i)-(iv).
	
	Fix $\beta < \mathfrak{c}$ and suppose that $S_{\xi, \gamma} \in [\omega]^{\omega}$ is defined for every $\gamma < \beta$ so that $S_{\xi, \delta} \subset^{*} S_{\xi, \epsilon}$ for all $\epsilon < \delta < \beta$ and $\xi \in J_1 \cup \bigcup_{n \in D} J_n$. Suppose also that we have constructed a group homomorphism $\phi_{\gamma}: (\mathbb{Q} / \mathbb{Z})^{(P_0 \times \omega)} \oplus \mathbb{Q}^{(P_1)} \to \mathbb{T}$ satisfying (i)-(iv). We will show that it is possible to choose $S_{\xi, \beta} \in [\omega]^{\omega}$ so that $S_{\xi, \beta} \subset^{*} S_{\xi, \gamma}$ for all $\gamma < \beta$ and $\xi \in J_1 \cup \bigcup_{n \in D} J_n$ and that it is also possible to construct a group homomorphism $\phi_{\beta} : (\mathbb{Q} / \mathbb{Z})^{(P_0 \times \omega)} \oplus \mathbb{Q}^{(P_1)} \to \mathbb{T}$ satisfying (i)-(iv).
	
	If $\beta$ is a successor ordinal --- say, $\beta = \gamma + 1$ --- put $R_{\xi, \beta} = S_{\xi, \gamma}$ for every $\xi \in J_1 \cup \bigcup_{n \in D} J_n$ and apply Lemma \ref{lem_ext_hom} to $\alpha = \beta$ and $R_{\xi} = R_{\xi, \beta}$. If $\beta$ is a limit ordinal then consider, for each $\xi \in J_1 \cup \bigcup_{n \in D} J_n$, the family $\{S_{\xi, \gamma} : \gamma < \beta\}$. By inductive hypothesis, this family has the SFIP and, since we are assuming $\mathfrak{p} = \mathfrak{c}$, it has a pseudointersection $R_{\xi, \beta}$. Then, apply Lemma \ref{lem_ext_hom} to $\alpha = \beta$ and $R_{\xi} = R_{\xi, \beta}$.
\end{proof}

\end{document}